\UseRawInputEncoding
\documentclass[11pt,twoside]{article}
\usepackage{amsmath,amsthm,amssymb,amscd,ascmac, amsfonts,fancybox}
\usepackage{mathrsfs}
\usepackage{url}
\allowdisplaybreaks[4]

\usepackage[dvips]{graphicx,color,psfrag}

\newcommand{\sn}{\mathrm{sn}\,}
\newcommand{\cn}{\mathrm{cn}\,}
\newcommand{\dn}{\mathrm{dn}\,}

\newcommand{\Zz}{\mathbb{Z}}
\newcommand{\tW}{\mathcal{W}}

\theoremstyle{plain}
\newtheorem{thm}{Theorem}

\newtheorem{lem}[thm]{Lemma}
\newtheorem{cor}[thm]{Corollary}

\theoremstyle{definition}

\newtheorem{rmk}[thm]{Remark}

\begin{document}
\title{The generalized Zwegers' $\mu$-function and transformation formulas for the bilateral basic hypergeometric series}
\author{Genki Shibukawa and Satoshi Tsuchimi}
\date{\empty}
\maketitle

\begin{abstract}
By applying Slater's transformation formulas for the bilateral basic hypergeometric series ${}_2\psi_{2}$, we derive three translation or connection formulas for the generalized Zwegers' $\mu$-function (``continuous $q$-Hermite function'') which was introduced by Shibukawa--Tsuchimi (SIGMA, 2023). 
From some Bailey's transformation formula of ${}_2\psi_{2}$, we also give a formula for the expression of the generalized Zwegers' $\mu$-function by a certain degenerate Very-Well-Poised bilateral basic hypergeometric series ${}_4\psi_{8}$. 
As an application of this new expression formula for the generalized Zwegers' $\mu$-function, we obtain some new $q$-expansions for elliptic functions and Ramanujan's mock theta functions.
\end{abstract}

\section{Introduction}
We denote the set of integers and the set of complex numbers by $\mathbb{Z}$ and $\mathbb{C}$ respectively. 
Let $i:=\sqrt{-1}$ be the imaginary unit, $\tau $ be a complex number $\mathrm{Im}(\tau)>0$, and $q:={\rm e}^{2\pi i\tau}$. 
For non-negative integers $r$ and $s$, we define the basic hypergeometric series ${}_r\phi_{s}$ and the bilateral basic hypergeometric series ${}_r\psi_{s}$ as follows:
\begin{align*}
{}_r\phi_{s}\left( \begin{matrix} a_1,\dots,a_r \\ b_1,\dots,b_s \end{matrix};q,x\right)
 &:=
 \sum_{n=0}^{\infty }\frac{(a_1,\dots, a_r)_n}{(b_1,\dots,b_s,q)_n}\left((-1)^{n }q^\frac{n(n-1)}{2}\right)^{s-r+1}x^{n }, \\
{}_r\psi_{s}\left( \begin{matrix} a_1,\dots, a_r\\b_1,\dots, b_s\end{matrix};q,x\right)
 &:=
 \sum_{n \in \mathbb{Z}}\frac{(a_1,\dots, a_r)_n}{(b_1,\dots, b_s)_n}\left((-1)^{n }q^\frac{n(n-1)}{2}\right)^{s-r}x^{n },
\end{align*}
where $a_{1},\ldots,a_{r},b_{1},\ldots,b_{s},x$ are appropriate complex numbers and 
\begin{align*}
(a_1,\dots,a_r)_{n }
   &=
   (a_1,\dots,a_r;q)_{n }
   :=
   (a_1;q)_{n}\cdots (a_r;q)_{n}, \qquad n \in \mathbb{Z}\cup \{\infty \}, \\
(a)_{n}
   &=
   (a;q)_{n}
   :=
   \frac{(a;q)_{\infty }}{(q^{n}a;q)_{\infty }}, \qquad n \in \mathbb{Z}, \\
(a)_{\infty }
   &=
   (a;q)_{\infty }
   :=
   \prod_{j=0}^{\infty }(1-aq^{j}).
\end{align*}

Zwegers \cite{Zw} introduced the following function which we call {\it{Zwegers' $\mu $-function}}:
\begin{align*}
\mu(u,v;\tau)
   &:=
   \frac{e^{\pi iu}}{\vartheta_{11}(v)}
   \frac{1}{1-e^{2\pi iu}}
   {}_1\psi_2\left(\begin{matrix}e^{2\pi iu}/q \\ 0,e^{2\pi iu} \end{matrix};q,qe^{2\pi iv}\right) \\
   &=
   \frac{e^{\pi iu}}{\vartheta_{11}(v)}\sum_{n \in \mathbb{Z}}\frac{(-1)^{n}e^{2\pi inv}q^{\frac{n(n+1)}{2}}}{1-e^{2\pi iu}q^n}, \qquad u,v \in \mathbb{C}\setminus \mathbb{Z}+\mathbb{Z}\tau ,
\end{align*}
where $\vartheta_{11}$ is the theta function defined by 
\begin{align*}
\vartheta_{11}(u)
   &:=
   \sum_{n \in \mathbb{Z}}e^{2\pi i(n+\frac{1}{2})(u+\frac{1}{2})+\pi i(n+\frac{1}{2})^{2}\tau}
   =-iq^{\frac{1}{8}}e^{-\pi iu}(q,e^{2\pi iu}, qe^{-2\pi iu};q )_{\infty }.
\end{align*}

Zwegers' $\mu$-function is a special function including many mock theta functions as its specializations. 
At least all mock theta functions discovered by Ramanujan are known to be obtained as linear combinations of some specializations for the $\mu$-functions and appropriate $q$-infinite products \cite[Appendix~A]{BFOR}. 
Zwegers showed that, by adding an appropriate non-holomorphic function, the $\mu$-function has properties of real analytic bivariate Jacobi (like) forms. 
These modular properties and his key idea of the non-holomorphic deformation for the $\mu$-function greatly influenced later investigations on mock theta functions and mock modular forms. 

On the other hand, the authors introduced a one parameter deformation of Zwegers' $\mu $-function, which we call the {\it{generalized Zwegers' $\mu $-function}}, as the following series: 
\begin{align}
\label{eq:def of gen mu 1}
\mu(u,v;\alpha,\tau)
   =
   \mu(u,v;\alpha)
   :=&
   \frac{e^{\pi i\alpha (u-v)}}{\vartheta_{11}(v)}
   \sum_{n \in \mathbb{Z}}(-1)^{n}e^{2\pi i(n+\frac{1}{2})v}q^{\frac{n(n+1)}{2}}
   \frac{\big( e^{2\pi iu}q^{n+1}\big)_{\infty }}{\big( e^{2\pi iu}q^{n-\alpha +1}\big)_{\infty }} \\
\label{eq:def of gen mu 2}
   =&
   -iq^{-\frac{1}{8}}\frac{(x/y)^\frac{\alpha}{2}}{(q)_{\infty }\theta(y)}\frac{(x)_\infty}{(x/a)_\infty}
   {}_1\psi_2\left(\begin{matrix}x/a \\ 0,x \end{matrix};q,y\right),
\end{align}
where 
$$
x:=e^{2\pi iu}, \quad y:=e^{2\pi iv}, \quad a:=q^{\alpha }=e^{2\pi i\alpha \tau }, \quad \theta (y):=(y,q/y)_{\infty }.
$$
For convenience, (\ref{eq:def of gen mu 1}) is often also written in the multiplicative notation as $\mu(x,y;a) = \mu(u,v;\alpha)$.

The generalized $\mu $-function is derived from the image of the composition of the $q$-Borel transformation $\mathcal{B}^{+}$ and $q$-Laplace transformation $\mathcal{L}^{+}$ of the formal solution at $x=0$:
$$
x^{\frac{\alpha }{2}}\widetilde{f}_{0}(x;a), \quad 
   \widetilde{f}_{0}(x;a)
   :=
   {}_2\phi_0\left(\begin{matrix}a,0 \\ - \end{matrix};q, {\frac{x}{a}}\right)
$$
of the following second-order linear $q$-difference equation of the Laplace type ($q$-Hermite-Weber equation):
\begin{align}\label{eq:q-Hermite--Weber}
[T_x^2-(1-xq)\sqrt{a}T_x-xq]f(x)=0.
\end{align}
Precisely, the following fact holds (for details, see \cite[Theorem~1.2]{ST}):
\begin{align}\label{eq:mu and q-Hermite--Weber}
f_{0}\big(e^{2\pi i(u-v)},-e^{2\pi iu};e^{2\pi i\tau \alpha }\big)
   =
   \mu(u,v;\alpha),
\end{align}
where 
$$
f_{0}(x,\lambda ;a)
   :=
    x^{\frac{\alpha }{2}}
    \mathcal{L}^{+}\circ\mathcal{B}^{+}\big(\widetilde{f}_{0}\big)(x,\lambda).
$$

Since the case of $\alpha =1$ is the original Zwegers' $\mu$-function:
$$
f_{0}\big(e^{2\pi i(u-v)},-e^{2\pi iu};e^{2\pi i\tau }\big)
   =
   \mu(u,v),
$$
Zweger's $\mu$-function (for this case, Garoufalidis--Wheeler \cite{GW} pointed out this fact independently) and the generalized $\mu$-function are characterized as one of the fundamental solutions which comes from the formal solution $x^{\frac{\alpha }{2}}\widetilde{f}_{0}(x;a)$ at $x=0$ for the $q$-Hermite--Weber ($q$-HW) equation.

Shibukawa--Tsuchimi gave fundamental formulas for the generalized $\mu$-function similar to the original $\mu$-function \cite[Theorem~1.3]{ST}, and also pointed out the generalized $\mu$-function satisfies the following recursion for the parameter $a = e^{2\pi i \alpha \tau}$ \cite[Theorem~1.3 (1.31)]{ST}:
$$
2\cos{\pi (u-v)}\mu(u,v;\alpha )
   =
   (1 - q^{- \alpha })\mu(u,v;\alpha + 1) + \mu(u,v;\alpha - 1),
$$
which coincides essentially with the $q$-Bessel equation. 
As a corollary, they proved if $\alpha $ is a non-positive integer $-k$ then the generalized $\mu$-function equals to the continuous $q$-Hermite polynomial \cite[Theorem~1.6]{ST}:
\begin{align}\label{eq:mu and CqH}
\mu(u,v;-k)
   =
   -iq^{-\frac{1}{8}}H_{k}(\cos{\pi (u-v)}\mid q),
\end{align}
where $H_{k}$ is the continuous $q$-Hermite polynomial of degree $k$:
$$
H_{k}(\cos{\pi (u-v)} \mid q)
   :=
   \sum_{l=0}^{k}
   \frac{(q)_{k}}{(q)_{l}(q)_{k-l}}
   e^{\pi i(k-2l)(u-v)}.
$$

Therefore, the generalized $\mu$-function $\mu(u,v;\alpha)$ is also regarded as a ``continuous $q$-Hermite function'' which is a continuous deformation of the continuous $q$-Hermite polynomial with the degree parameter $k$, and the original $\mu$-function is regarded as the continuous $q$-Hermite polynomial of ``degree $-1$''.

The continuous $q$-Hermite polynomials are a typical example of $q$-orthogonal polynomials, and have some interesting properties \cite[Chapter~13.1]{I}. 
Thus the generalized $\mu$-function which includes the original Zwegers' $\mu$-function as a special case is important in mock theta functions and mock modular forms, but is also a fundamental object in special functions.

However, Shibukawa--Tsuchimi needed some results in \cite{Zh} or non-publish results in \cite{Oh2} to prove the translation formula \cite[Theorem~1.3 (1.27)]{ST}:
\begin{align}
\mu(u+z,v+z;\alpha)
   &=
   \Phi(u,v;\alpha)\mu(v,u;\alpha) \nonumber \\
   & \quad -
   \frac{i(q^{\alpha })_{\infty }(q)_{\infty}^{2}q^{\frac{1-4\alpha}{8}}\vartheta_{11}(z)\vartheta_{11}(u+v+z-\alpha \tau)}{\vartheta_{11}(u)\vartheta_{11}(v-\alpha \tau)\vartheta_{11}(u+z-\alpha \tau)\vartheta_{11}(v+z)} \nonumber \\
\label{eq:general mu trans 1}
   & \quad \cdot e^{\pi i(\alpha-1)(u-v)}\, {}_1\phi_{1}\left(\begin{matrix}q^{1-\alpha} \\ 0 \end{matrix};q, e^{-2\pi i(u-v)}q\right), \\
\Phi(u,v;\alpha)
   &:=
   \frac{\vartheta_{11}(v-\alpha\tau)\vartheta_{11}(u)}{\vartheta_{11}(u-\alpha\tau)\vartheta_{11}(v)}e^{2\pi i\alpha(u-v)}, 
\end{align}
or its variation \cite[Corollary~3.1 (3.3)]{ST}:
\begin{align}
\label{eq:general mu trans 2}
iq^\frac{1}{8}\mu(u,v;\alpha)
   &=
    \Phi(u,v;\alpha)j(u-v;\alpha)+j(v-u;\alpha), \\
j(w;\alpha)
   &:=
   iq^{\frac{1}{8}}
   \frac{(q)_{\infty}}{(q^{1-\alpha})_{\infty}}
   \frac{e^{\pi i(1-\alpha)w}}{\vartheta_{11}(w)}
   {}_1\phi_{1}\left(\begin{matrix}q^{1-\alpha} \\ 0 \end{matrix} ; q,e^{2\pi iw}q \right). \nonumber 
\end{align}
On the other hand, the generalized $\mu$-function has the expression formula by the bilateral series ${}_2\psi_{2}$ (\ref{eq:mu 2psi2}). 
Since the bilateral series ${}_2\psi_{2}$ is a bilateral analogue of Heine's $q$-hypergeometric function ${}_2\phi_{1}$, it has some nice properties. 
First, the bilateral basic hypergeometric function ${}_2\psi_{2}$ is the simplest case of the bilateral series ${}_r\psi_{r}$ for which the non-trivial Slater's transformation formula (\ref{eq:Slater A r}) holds. 
Second, Bailey's transformations (\ref{eq:Bailey trans0}), (\ref{eq:Bailey trans1}), (\ref{eq:Bailey trans2}), and (\ref{eq:Bailey trans3}) hold, which is a bilateral version of the following Heine's transformation:
\begin{gather*}
{_{2}\phi _1}\left(\begin{matrix}a,b \\ c \end{matrix};q,x\right)
 =
 \frac{(ax,c/a)_{\infty}}{(x,c)_{\infty}}\,
 {_{2}\phi _1}\left(\begin{matrix}a,abx/c \\ ax \end{matrix};q,\frac{c}{a}\right).
\end{gather*}

In this paper, by applying Slater's transformation formulas (\ref{eq:Slater A 2 1}), (\ref{eq:Slater A 2 2}) and (\ref{eq:Slater A 2 3}), we give unified proofs of the following three formulas. 
\begin{thm}
\label{thm:main theorem1}
For the free parameters $x^{\prime}$ and $y^{\prime}$, we have 
\begin{align}
\mu{(x,y;a)}
   &=
   x^{\prime }\frac{\theta (xyy^{\prime }/q^{2}a)\theta (yx^{\prime })\theta (xx^{\prime }/a)\theta (q/y^{\prime })}{\theta \left(xyx^{\prime }y^{\prime }/q^{2}a\right)\theta (y)\theta (x/a)\theta (y^{\prime }/x^{\prime })}
   \mu{(xx^{\prime },yx^{\prime };a)} \nonumber \\
\label{eq:mu trans 1}   
   & \quad +
   y^{\prime }\frac{\theta (xyx^{\prime }/q^{2}a)\theta (xy^{\prime })\theta (yy^{\prime }/a)\theta (q/x^{\prime })}{\theta \left(xyx^{\prime }y^{\prime }/q^{2}a\right)\theta (y)\theta (x/a)\theta (x^{\prime }/y^{\prime })}
   \mu{(xy^{\prime },yy^{\prime };a)} \\
   &=
   x^{\prime }\frac{\theta (x/q^{2})\theta (yx^{\prime })\theta (xx^{\prime }/a)\theta (qy/a)}{\theta \left(xx^{\prime }/q^{2}\right)\theta (y)\theta (x/a)\theta (a/yx^{\prime })}
   \mu{(xx^{\prime },yx^{\prime };a)} \nonumber \\
\label{eq:mu trans 2}
   & \quad -iq^{-\frac{1}{8}}\left(\frac{x}{y}\right)^{\frac{\alpha }{2}}
   \frac{a}{y}\frac{\theta (xyx^{\prime }/q^{2}a)\theta (q/x^{\prime })(a,qy/x)_{\infty }}{\theta \left(xx^{\prime }/q^{2}\right)\theta (y)\theta (x/a)\theta (yx^{\prime }/a)} 
   {_{0}\phi _1}\left(\begin{matrix} - \\ qy/x \end{matrix};q,\frac{q^{2}y}{ax}\right) \\
   &=
   -iq^{-\frac{1}{8}}\left(\frac{x}{y}\right)^{\frac{\alpha }{2}}
   \frac{a}{x}\frac{\theta (x/q^{2})\theta (qy/a)(a,qx/y)_{\infty }}{\theta \left(a/q^{2}\right)\theta (y)\theta (x/a)\theta (x/y)}
   {_{0}\phi _1}\left(\begin{matrix} - \\ qx/y \end{matrix};q,\frac{q^{2}x}{ay}\right) \nonumber \\
\label{eq:mu trans 3}
   & \quad -iq^{-\frac{1}{8}}\left(\frac{x}{y}\right)^{\frac{\alpha }{2}}
   \frac{a}{y}\frac{\theta (y/q^{2})\theta (qx/a)(a,qy/x)_{\infty }}{\theta \left(a/q^{2}\right)\theta (y)\theta (x/a)\theta (y/x)}
   {_{0}\phi _1}\left(\begin{matrix} - \\ qy/x \end{matrix};q,\frac{q^{2}y}{ax}\right).
\end{align}
\end{thm}
Theorem~\ref{thm:main theorem1} is regarded as some connection formulas for certain fundamental solutions of the $q$-HW equation or $q$-Bessel equation. 
The first formula (\ref{eq:mu trans 1}) is a connection formula between $\mu{(x,y;a)}$ which is a fundamental solution of the $q$-HW equation, and other two independent solutions $\mu{(xx^{\prime },yx^{\prime };a)}$ and $\mu{(xy^{\prime },yy^{\prime };a)}$. 
Similarly, the second formula (\ref{eq:mu trans 2}) is a connection formula between $\mu{(x,y;a)}$ and two fundamental solutions $\mu{(xx^{\prime },yx^{\prime };a)}$ and ${_{0}\phi _1}$ of the $q$-HW equation. 
Since $\mu{(x,y;a)}$ satisfies the $q$-Bessel equation for the parameter $a$, there is also a connection formula between the generalized $\mu$-function and two fundamental solutions of the $q$-Bessel equation, which is the third formula (\ref{eq:mu trans 3}).

We remark that (\ref{eq:mu trans 2}) and (\ref{eq:mu trans 3}) are equivalent to (\ref{eq:general mu trans 1}) and (\ref{eq:general mu trans 2}) respectively, and that our method is a direct proof of (\ref{eq:general mu trans 1}) and (\ref{eq:general mu trans 2}), unlike the proof in \cite{ST}, without using the $q$-Borel and $q$-Laplace transformations.

Further, by Bailey's transformation formula (\ref{eq:Slater BC 8 2psi2 2}), we also obtain the following expression formula for the generalized $\mu$-function.
\begin{thm}
\label{thm:main theorem2}
We have
\begin{align}
\mu(x,y;a)
   &=
   -iq^{-\frac{1}{8}}
   \left(\frac{x}{y}\right)^{\frac{\alpha}{2}}
   \frac{(x,y,aq/x,aq/y)_{\infty }}{(q)_{\infty }\theta(y)\theta(x/a)\theta(xy/aq)} \nonumber \\
\label{eq:generalized mu BC1}
   & \quad \cdot 
   \left(1-\frac{xy}{aq}\right){_{4}\psi _8}\left(\begin{matrix} \sqrt{\frac{qxy}{a}},-\sqrt{\frac{qxy}{a}},x/a,y/a \\ \sqrt{\frac{xy}{aq}},-\sqrt{\frac{xy}{aq}},y,x,0,0,0,0 \end{matrix};q,\frac{x^{2}y^{2}}{aq}\right) \\   
   &=
   -iq^{-\frac{1}{8}}
   \left(\frac{x}{y}\right)^{\frac{\alpha}{2}}
   \frac{(x,y,aq/x,aq/y)_{\infty }}{(q)_{\infty }\theta(y)\theta(x/a)\theta(xy/aq)} \nonumber \\
\label{eq:generalized mu BC2}
   & \quad \cdot 
   \sum_{n \in \mathbb{Z}}
      \left(1-\frac{xy}{a}q^{2n-1}\right)
      \frac{(x/a,y/a)_n}{(x,y)_n}q^{2n^{2}-3n}\left(\frac{x^{2}y^{2}}{a}\right)^{n}.
\end{align}
If the case of $a=q$, we have
\begin{align}
\mu(x,y)
   &=
   iq^{-\frac{1}{8}}
   \frac{\sqrt{xy}}{\left(1-\frac{x}{q}\right)(q-y)}
   \frac{1}{(q)_{\infty }\theta(xy/q^{2})} \nonumber \\
\label{eq:mu BC1}
   & \quad \cdot 
   \left(1-\frac{xy}{q^{2}}\right){_{4}\psi _8}\left(\begin{matrix} \sqrt{xy},-\sqrt{xy},x/q,y/q \\ \frac{\sqrt{xy}}{q},-\frac{\sqrt{xy}}{q},y,x,0,0,0,0 \end{matrix};q,\frac{x^{2}y^{2}}{q^{2}}\right) \\   
\label{eq:mu BC2}
   &=
   iq^{-\frac{1}{8}}\frac{\sqrt{xy}}{(q)_{\infty }\theta(xy)}
   \sum_{n \in \mathbb{Z}}\frac{1-xyq^{2n}}{(1-xq^n)(1-yq^n)}x^{2n}y^{2n}q^{2n^{2}} \\
\label{eq:mu BC3}
   &=
   iq^{-\frac{1}{8}}\frac{\sqrt{xy}}{(q)_{\infty }\theta(xy)}
   \left\{\sum_{n \in \mathbb{Z}}\frac{x^{2n}y^{2n}q^{2n^{2}}}{(1-xq^n)(1-yq^n)}
   -\sum_{n \in \mathbb{Z}}\frac{x^{-2n}y^{-2n}q^{2n^{2}}}{(1-x^{-1}q^n)(1-y^{-1}q^n)}\right\}.
\end{align}
\end{thm}
As a corollary, we have some new $q$-expansion formulas for elliptic functions (Corollary~\ref{cor: mu elliptic functions}) and Ramanujan's mock theta functions (Appendix~A).

The contents of this article are as follows. 
In Section~2, we state some basic properties of the basic hypergeometric functions ${}_r\phi_{s}$ and ${}_r\psi_{s}$. 
Especially, we introduce certain Slater's and Bailey's transformation formulas and the expression formulas of the generalized $\mu$-function by the bilateral basic hypergeometric series, which are necessary for the proofs of our main results. 
In Section~3, we prove Theorem~\ref{thm:main theorem1} and Theorem~\ref{thm:main theorem2}, and give some corollaries and remarks of our main results. 
Finally, in Appendix, we list all mock theta functions discovered by Ramanujan, and give their curious expressions following from the formulas (\ref{eq:mu BC1}) and (\ref{eq:mu BC2}).

\section{Some transformation formulas for the series ${}_2\psi_{2}$ and bilateral basic hypergeometric expressions of $\mu(x,y;a)$}
Refer to \cite[Chapter~3]{AB} and \cite[Chapter~5]{GR} for the details in this section. 
Since $|q|<1$ ($\mathrm{Im}(\tau)>0$), if $s<r$ then the bilateral series ${}_r\psi_{s}$ diverges for all $x \in \mathbb{C}$ and the series ${}_{r+1}\phi_{s}$ diverges for all $x\not=0 \in \mathbb{C}$.  
If $r<s$, the series ${}_r\psi_{s}$ converges when $\left|\frac{b_{1}\cdots b_{s}}{a_{1}\cdots a_{r}}\right| < |x|$ and the series ${}_{r+1}\phi_{s}$ converges for any complex number $x$. 
If $r=s$, the bilateral series ${}_r\psi_{r}$ converges when 
\begin{align}
\label{eq:conv cond r psi r}
\left|\frac{b_{1}\cdots b_{s}}{a_{1}\cdots a_{r}}\right|<|x|<1
\end{align}
and the series ${}_{r+1}\phi_{r}$ converges when $|x| < 1$. 
Throughout this paper, we always assume (\ref{eq:conv cond r psi r}).

By the definition, we have
\begin{align}
\label{eq:psi inversion r}
{_{r}\psi _r}\left(\begin{matrix} a_{1},\ldots, a_{r} \\ b_{1},\ldots, b_{r} \end{matrix};q,x\right)
   &=
   {_{r}\psi _r}\left(\begin{matrix} q/b_{1},\ldots, q/b_{r} \\ q/a_{1},\ldots, q/a_{r} \end{matrix};q,\frac{b_{1}\cdots b_{r}}{a_{1}\cdots a_{r}x}\right).
\end{align}
Particular, 
\begin{align}
\label{eq:psi inversion 2}
{_{2}\psi _2}\left(\begin{matrix} a_{1},a_{2} \\ b_{1},b_{2} \end{matrix};q,x\right)
   &=
   {_{2}\psi _2}\left(\begin{matrix} q/b_{1},q/b_{2} \\ q/a_{1},q/a_{2} \end{matrix};q,\frac{b_{1}b_{2}}{a_{1}a_{2}x}\right).
\end{align}

We mention two kinds of Slater's transformation formulas (1952). 
\begin{lem}[{\cite[Chapter~5 (5.4.3), (5.5.2)]{GR}}, {\cite[(4)]{S}}]
\label{lem:Slater trans formulas}
For generic complex parameters $a_{1},\ldots,a_{r}$, $b_{1},\ldots,b_{r}$, $c_{1},\ldots,c_{r}$, we have
\begin{align}
& \theta (dxq)\prod_{j=1}^{r}\frac{(b_{j},q/a_{j})_{\infty }}{\theta (c_{j})}
   {}_r\psi_{r}\left( \begin{matrix} a_1,\dots, a_r \\ b_1,\ldots, b_r\end{matrix};q,x\right) \nonumber \\
\label{eq:Slater A r}
   &=
   \sum_{m=1}^{r}\frac{\theta (c_{m}dx)}{\theta (c_{m})}(c_{m}/a_{m})_{\infty }
   \prod_{1\leq j\not=m \leq r}\frac{(c_{m}/a_{j},b_{j}q/c_{m})_{\infty }}{\theta (c_{m}/c_{j})}
   {}_r\psi_{r}\left( \begin{matrix} a_{1}q/c_{m},\ldots, a_{r}q/c_{m} \\ b_{1}q/c_{m},\dots, b_{r}q/c_{m}\end{matrix};q,x\right),
\end{align}
where
$$
d:=\frac{a_{1}\cdots a_{r}}{c_{1}\cdots c_{r}}.
$$
Similarly, for generic complex parameters $a, a_{1},\ldots,a_{r}, b_{1},\ldots,b_{2r+2}$, we have
\begin{align}
& \frac{1-a}{\theta (a)}
      \frac{\prod_{1\leq j\leq 2r+2}(q/b_{j},aq/b_{j})_{\infty }}{\prod_{1\leq k\leq r}\theta (a_{k})\theta (a_{k}/a)}
      {}_{2r+4}\psi_{2r+4}\left( \begin{matrix} \sqrt{a}q,-\sqrt{a}q,b_{1},\ldots, b_{2r+2} \\ \sqrt{a},-\sqrt{a},aq/b_{1},\ldots, aq/b_{2r+2}\end{matrix};q,\frac{a^{r-1}q^{r-2}}{b_{1}\cdots b_{2r+2}}\right) \nonumber \\
   &=
   \sum_{m=1}^{r}
      \frac{1-a_{m}^{2}/a}{\theta (a_{m})\theta (a_{m}/a)\theta (a_{m}^{2}/a)}
      \frac{\prod_{1\leq j\leq 2r+2}(a_{m}q/b_{j},aq/a_{m}b_{j})_{\infty }}{\prod_{1\leq k\not= m \leq r}\theta (a_{k}/a_{m})\theta (a_{k}a_{m}/a)} \nonumber \\
\label{eq:Slater BC r}
   & \quad \cdot 
   {}_{2r+4}\psi_{2r+4}\left( \begin{matrix} a_{m}q/\sqrt{a},-a_{m}q/\sqrt{a},a_{m}b_{1}/a,\ldots, a_{m}b_{2r+2}/a \\ a_{m}/\sqrt{a},-a_{m}/\sqrt{a},a_{m}q/b_{1},\ldots, a_{m}q/b_{2r+2}\end{matrix};q,\frac{a^{r-1}q^{r-2}}{b_{1}\cdots b_{2r+2}}\right).
\end{align}
\end{lem}
\begin{rmk}
{\rm{(1)}} From some specializations of the case of $r=1$ for (\ref{eq:Slater A r}) and (\ref{eq:Slater BC r}), we obtain the well-known Ramanujan's ${_{1}\psi _1}$ summation:
$$
{_{1}\psi _1}\left(\begin{matrix} a \\ b \end{matrix};q,x\right)
 =
 \frac{(ax,q/xz,q,b/a)_{\infty}}{(x,b/ax,b,q/a)_{\infty}}
$$
and Bailey's ${_{6}\psi _6}$ summation:
\begin{align}
& {_{6}\psi _6}\left(\begin{matrix} \sqrt{a}q,-\sqrt{a}q,b,c,d,e \\ \sqrt{a},-\sqrt{a},aq/b,aq/c,aq/d,aq/e \end{matrix};q,\frac{qa^{2}}{bcde}\right) \nonumber \\
\label{eq:Bailey's summation}
 & \quad =
 \frac{(aq,aq/bc,aq/bd,aq/be,aq/cd,aq/ce,aq/de,q,q/a)_{\infty}}{(aq/b,aq/c,aq/d,aq/e,q/b,q/c,q/d,q/e,qa^{2}/bcde)_{\infty}}
\end{align}
respectively.\\
{\rm{(2)}} More generally, Slater's transformation formulas (\ref{eq:Slater A r}) and (\ref{eq:Slater BC r}) correspond to the simplest case of the connection formulas for the Jackson integral of type $A$ \cite[Theorem~1.2]{IN2} and type $BC$ \cite[Theorem~1.5]{IN1} respectively. 
We remark that Ito--Noumi's proofs of their connection formulas for some Jackson integrals by an elliptic analogue of Lagrange interpolation functions include elegant proofs of Slater's transformation formulas (\ref{eq:Slater A r}) and (\ref{eq:Slater BC r}).
\end{rmk}

In this paper, the case of $r=2$ is important. 
We need the following three variations on (\ref{eq:Slater A r}) for $r=2$ to prove the formulas (\ref{eq:mu trans 1}), (\ref{eq:mu trans 2}), and (\ref{eq:mu trans 3}). 
\begin{lem}
\label{lem:Slater 2psi2 lem}
For generic complex parameters $a_{1},a_{2}$, $b_{1},b_{2}$, $c_{1},c_{2}$, we obtain 
\begin{align}
& {_{2}\psi _2}\left(\begin{matrix} a_{1},a_{2} \\ b_{1},b_{2} \end{matrix};q,x\right) \nonumber \\
   &=
   \frac{q}{c_{1}}\frac{\theta (a_{1}a_{2}x/qc_{2})\theta (c_{2})}{\theta \left(a_{1}a_{2}x/c_{1}c_{2}\right)\theta (c_{1}/c_{2})}\frac{(c_{1}/a_{1},c_{1}/a_{2},qb_{2}/c_{1})_{\infty }}{(q/a_{1},q/a_{2},b_{1},b_{2})_{\infty }}
   {_{2}\psi _2}\left(\begin{matrix} qa_{1}/c_{1},qa_{2}/c_{1} \\ qb_{1}/c_{1},qb_{2}/c_{1} \end{matrix};q,x\right) \nonumber \\
\label{eq:Slater A 2 1}
   & \quad +
   \frac{q}{c_{2}}\frac{\theta (a_{1}a_{2}x/qc_{1})\theta (c_{1})}{\theta \left(a_{1}a_{2}x/c_{1}c_{2}\right)\theta (c_{2}/c_{1})}\frac{(c_{2}/a_{1},c_{2}/a_{2},qb_{1}/c_{2})_{\infty }}{(q/a_{1},q/a_{2},b_{1},b_{2})_{\infty }}
   {_{2}\psi _2}\left(\begin{matrix} qa_{1}/c_{2},qa_{2}/c_{2} \\ qb_{1}/c_{2},qb_{2}/c_{2} \end{matrix};q,x\right) \\
   &=
   \frac{q}{c_{1}}\frac{\theta (a_{1}a_{2}x/qc_{2})\theta (c_{2})}{\theta \left(a_{1}a_{2}x/c_{1}c_{2}\right)\theta (c_{1}/c_{2})}\frac{(c_{1}/a_{1},c_{1}/a_{2},qb_{2}/c_{1})_{\infty }}{(q/a_{1},q/a_{2},b_{1},b_{2})_{\infty }}
   {_{2}\psi _2}\left(\begin{matrix} qa_{1}/c_{1},qa_{2}/c_{1} \\ qb_{1}/c_{1},qb_{2}/c_{1} \end{matrix};q,x\right) \nonumber \\
\label{eq:Slater A 2 2}
   & \quad +
   \frac{q}{c_{2}}\frac{\theta (a_{1}a_{2}x/qc_{1})\theta (c_{1})}{\theta \left(a_{1}a_{2}x/c_{1}c_{2}\right)\theta (c_{2}/c_{1})}\frac{(c_{2}/a_{1},c_{2}/a_{2},qb_{1}/c_{2})_{\infty }}{(q/a_{1},q/a_{2},b_{1},b_{2})_{\infty }}
   {_{2}\psi _2}\left(\begin{matrix} c_{2}/b_{1},c_{2}/b_{2} \\ c_{2}/a_{1},c_{2}/a_{2} \end{matrix};q,\frac{b_{1}b_{2}}{a_{1}a_{2}x}\right) \\
   &=
   \frac{q}{c_{1}}\frac{\theta (a_{1}a_{2}x/qc_{2})\theta (c_{2})}{\theta \left(a_{1}a_{2}x/c_{1}c_{2}\right)\theta (c_{1}/c_{2})}\frac{(c_{1}/a_{1},c_{1}/a_{2},qb_{2}/c_{1})_{\infty }}{(q/a_{1},q/a_{2},b_{1},b_{2})_{\infty }}
   {_{2}\psi _2}\left(\begin{matrix} c_{1}/b_{1},c_{1}/b_{2} \\ c_{1}/a_{1},c_{1}/a_{2} \end{matrix};q,\frac{b_{1}b_{2}}{a_{1}a_{2}x}\right) \nonumber \\
\label{eq:Slater A 2 3}
   & \quad +
   \frac{q}{c_{2}}\frac{\theta (a_{1}a_{2}x/qc_{1})\theta (c_{1})}{\theta \left(a_{1}a_{2}x/c_{1}c_{2}\right)\theta (c_{2}/c_{1})}\frac{(c_{2}/a_{1},c_{2}/a_{2},qb_{1}/c_{2})_{\infty }}{(q/a_{1},q/a_{2},b_{1},b_{2})_{\infty }}
   {_{2}\psi _2}\left(\begin{matrix} c_{2}/b_{1},c_{2}/b_{2} \\ c_{2}/a_{1},c_{2}/a_{2} \end{matrix};q,\frac{b_{1}b_{2}}{a_{1}a_{2}x}\right).
\end{align}
\end{lem}

Next, we introduce Bailey's transformation for ${}_{2}\psi_{2}$ (1950). 
\begin{lem}[{\cite[(3.2.4)]{AB}}, {\cite[(3.2)]{B1}}, {\cite[p.~147, Exercise~5.11]{GR}}]
\label{lem:Bailey trans VWP}
We have
\begin{align}
{_{2}\psi _2}\left(\begin{matrix} e,f \\ aq/c,aq/d \end{matrix};q,\frac{aq}{ef}\right)
   &=
   \frac{(q/c,q/d,aq/e,aq/f)_{\infty }}{(aq,q/a,aq/cd,aq/ef)_{\infty }} \nonumber \\
\label{eq:Slater BC 8 2psi2}
   & \quad \cdot 
   {_{6}\psi _8}\left(\begin{matrix} \sqrt{a}q,-\sqrt{a}q,c,d,e,f \\ \sqrt{a},-\sqrt{a},aq/c,aq/d,aq/e,aq/f,0,0 \end{matrix};q,\frac{a^{3}q^{2}}{cdef}\right),
\end{align}
and
\begin{align}
{_{2}\psi _2}\left(\begin{matrix} a,b \\ c,d \end{matrix};q,x\right)
   &=
   \frac{(ax,bx,qc/abx,qd/abx)_{\infty }}{(x,abx,q^{2}/abx,cd/abx)_{\infty }} \nonumber \\
\label{eq:Slater BC 8 2psi2 2}
   & \quad \cdot 
   {_{6}\psi _8}\left(\begin{matrix} \sqrt{qabx},-\sqrt{qabx},abx/c,abx/d,a,b \\ \sqrt{abx/q},-\sqrt{abx/q},c,d,bx,ax,0,0 \end{matrix};q,\frac{cdx}{q}\right).
\end{align}
\end{lem}
\begin{rmk}
\label{rmk:BC to A}
For $r=2$, by (\ref{eq:Slater BC 8 2psi2 2}), Slater's transformation formula (\ref{eq:Slater A 2 1}) is derived from (\ref{eq:Slater BC r}) in $r=2$:
\begin{align}
& \frac{1-a}{\theta (a)}
      \frac{(q/b_{1},aq/b_{1}, \ldots ,q/b_{6},aq/b_{6})_{\infty }}{\theta (a_{1})\theta (a_{1}/a)\theta (a_{2})\theta (a_{2}/a)}
   {}_{8}\psi_{8}\left( \begin{matrix} \sqrt{a}q,-\sqrt{a}q,b_{1},\ldots ,b_{6} \\ \sqrt{a},-\sqrt{a},aq/b_{1},\ldots ,aq/b_{6}\end{matrix};q,\frac{a^{3}q^{2}}{b_{1}b_{2}b_{3}b_{4}b_{5}b_{6}}\right) \nonumber \\
   &=
   \left(1-\frac{a_{1}^{2}}{a}\right)
      \frac{(a_{1}q/b_{1},aq/a_{1}b_{1},\ldots ,a_{1}q/b_{6},aq/a_{1}b_{6})_{\infty }}{\theta (a_{1})\theta (a_{1}/a)\theta (a_{1}^{2}/a)\theta (a_{2}/a_{1})\theta (a_{2}a_{1}/a)} \nonumber \\
   & \quad \cdot 
   {}_{8}\psi_{8}\left( \begin{matrix} a_{1}q/\sqrt{a},-a_{1}q/\sqrt{a},a_{1}b_{1}/a,\ldots, a_{1}b_{6}/a \\ a_{1}/\sqrt{a},-a_{1}/\sqrt{a},a_{1}q/b_{1},\ldots, a_{1}q/b_{6}\end{matrix};q,\frac{a^{3}q^{2}}{b_{1}b_{2}b_{3}b_{4}b_{5}b_{6}}\right) \nonumber \\
   & \quad +
   \left(1-\frac{a_{2}^{2}}{a}\right)
      \frac{(a_{2}q/b_{1},aq/a_{2}b_{1},\ldots ,a_{2}q/b_{6},aq/a_{2}b_{6})_{\infty }}{\theta (a_{2})\theta (a_{2}/a)\theta (a_{2}^{2}/a)\theta (a_{1}/a_{2})\theta (a_{2}a_{1}/a)} \nonumber \\
\label{eq:Slater BC 4}
   & \quad \quad \cdot 
   {}_{8}\psi_{8}\left( \begin{matrix} a_{2}q/\sqrt{a},-a_{2}q/\sqrt{a},a_{2}b_{1}/a,\ldots, a_{2}b_{6}/a \\ a_{2}/\sqrt{a},-a_{2}/\sqrt{a},a_{2}q/b_{1},\ldots, a_{2}q/b_{6}\end{matrix};q,\frac{a^{3}q^{2}}{b_{1}b_{2}b_{3}b_{4}b_{5}b_{6}}\right).
\end{align}
In fact, by taking the limit 
\begin{align*}
a \to \frac{a_{1}a_{2}x}{q}, \quad
a_{1} \to \frac{a_{1}a_{2}x}{c_{1}}, \quad
a_{2} \to \frac{a_{1}a_{2}x}{c_{2}}, \quad
b_{1} \to a_{1}, \quad
b_{2} \to a_{2}, \quad
b_{3} \to \frac{a_{1}a_{2}x}{b_{1}}, \quad
b_{4} \to \frac{a_{1}a_{2}x}{b_{2}},
\end{align*}
and $b_{7}, b_{8} \to \infty$ in (\ref{eq:Slater BC 4}) and using (\ref{eq:Slater BC 8 2psi2 2}), we have Slater's transformation formula (\ref{eq:Slater A 2 1}). 
For any $r>2$, since no result analogous to (\ref{eq:Slater BC 8 2psi2 2}) is known, we could not derive (\ref{eq:Slater BC r}) to (\ref{eq:Slater A r}) at the present stage. 
\end{rmk}

Finally, we mention some expressions of the generalized $\mu$-function by the bilateral basic hypergeometric series ${}_1\psi_{2}$, ${}_2\psi_{2}$ and ${}_0\psi_{2}$. 
\begin{lem}[{[ST, Theorem 1.4]}]
The generalized $\mu$-function has the following expressions:
\begin{align}
\mu(x,y;a)
 &=
 -iq^{-\frac{1}{8}}\frac{(x/y)^\frac{\alpha}{2}}{(q)_{\infty }\theta(x/a)}\frac{(aq/y)_{\infty }}{(q/y)_{\infty }}{}_1\psi_{2}\left(\begin{matrix}y/a\\0,y\end{matrix};q,x\right) \nonumber \\
\label{eq:mu 2psi2}
 &=
 -iq^{-\frac{1}{8}}\left(\frac{x}{y}\right)^\frac{\alpha}{2}\frac{(a,aq/x,aq/y)_{\infty }}{(q)_{\infty }\theta(y)\theta(x/a)}{}_2\psi_{2}\left(\begin{matrix}x/a,y/a\\0,0\end{matrix};q,a\right) \\
 &=
 -iq^{-\frac{1}{8}}\left(\frac{x}{y}\right)^\frac{\alpha}{2}\frac{(a,x,y)_{\infty }}{(q)_{\infty }\theta(y)\theta(x/a)}{}_0\psi_{2}\left(\begin{matrix}-\\x,y\end{matrix};q,\frac{xy}{a}\right). \nonumber 
\end{align}
\end{lem}
These bilateral series expressions follow from other types of Bailey's transformation formulas for ${}_{2}\psi_{2}$ \cite[(3.2.3)]{AB}, \cite[(2.3)]{B1}, \cite[Chapter~5, Exercise~5.20]{GR}:\begin{align}
{_{2}\psi _2}\left(\begin{matrix}a,b \\ c,d \end{matrix};q,x\right)
\label{eq:Bailey trans0}
 &=
 \frac{(ax,c/a,d/b,qc/abx)_{\infty}}{(x,c,q/b,cd/abx)_{\infty}}
 {_{2}\psi _2}\left(\begin{matrix}a,abx/c \\ ax,d \end{matrix};q,\frac{c}{a}\right) \\
\label{eq:Bailey trans1}
 &=
 \frac{(bx,d/b,c/a,qd/abx)_{\infty}}{(x,d,q/a,cd/abx)_{\infty}}
 {_{2}\psi _2}\left(\begin{matrix}b,abx/d \\ bx,c \end{matrix};q,\frac{d}{b}\right) \\
\label{eq:Bailey trans2}
 &=
 \frac{(ax,d/a,c/b,qd/abx)_{\infty}}{(x,d,q/b,cd/abx)_{\infty}}
 {_{2}\psi _2}\left(\begin{matrix}a,abx/d \\ ax,c \end{matrix};q,\frac{d}{a}\right) \\
\label{eq:Bailey trans3}
 &=
 \frac{(bx,c/b,d/a,qc/abx)_{\infty}}{(x,c,q/a,cd/abx)_{\infty}}
 {_{2}\psi _2}\left(\begin{matrix}b,abx/c \\ bx,d \end{matrix};q,\frac{c}{b}\right).
\end{align}

\section{Proofs of main results}
First, we prove Theorem~\ref{thm:main theorem1} by taking some appropriate limits in Lemma~\ref{lem:Slater 2psi2 lem} and applying (\ref{eq:mu 2psi2}).

\noindent
{\bf{Proof of Theorem~\ref{thm:main theorem1}}} 
We prove the formula (\ref{eq:mu trans 1}) by taking the limit
$$
a_{1} \to x/a, \quad a_{2} \to y/a, \quad c_{1} \to q/x^{\prime}, \quad c_{2} \to q/y^{\prime}, \quad x \to a
$$
and $b_{1},b_{2} \to 0$ in (\ref{eq:Slater A 2 1}), and by applying (\ref{eq:mu 2psi2}). 

If we take the limit
$$
a_{1} \to x/a, \quad a_{2} \to y/a, \quad c_{1} \to q/x^{\prime}, \quad c_{2} \to qy/a
$$
and $b_{1},b_{2} \to 0$ in (\ref{eq:Slater A 2 2}), we have
\begin{align}
{_{2}\psi _2}\left(\begin{matrix} x/a,y/a \\ 0,0 \end{matrix};q,a\right)
   &=
   x^{\prime }\frac{\theta (x/q)\theta (qy/a)}{\theta (xx^{\prime}/q^{2})\theta (a/yx^{\prime })}
   \frac{(qa/xx^{\prime }, qa/yx^{\prime })_{\infty }}{(qa/x,qa/y)_{\infty }}
   {_{2}\psi _2}\left(\begin{matrix} xx^{\prime }/a,yx^{\prime }/a \\ 0,0 \end{matrix};q,a\right) \nonumber \\
   & \quad +
   \frac{a}{y}x^{\prime }\frac{\theta (xyx^{\prime }/q^{2}a)\theta (q/x^{\prime })}{\theta (xx^{\prime}/q^{2})\theta (yx^{\prime }/a)}
   \frac{(qy/x,q)_{\infty }}{(qa/x,qa/y)_{\infty }} \nonumber \\
\label{eq:limit to q-Bessel}
   & \quad \quad \cdot 
   \lim_{b_{1},b_{2} \to 0}{_{2}\phi _1}\left(\begin{matrix} qy/ab_{1},qy/ab_{2} \\ qy/x \end{matrix};q,\frac{b_{1}b_{2}a}{xy}\right).
\end{align}
The first term on the right side of (\ref{eq:limit to q-Bessel}) is, from (\ref{eq:mu 2psi2}), essentially the generalized $\mu$-function $\mu(xx^{\prime},yx^{\prime};a)$.  
The limit of the second term of the right side in (\ref{eq:limit to q-Bessel}) is equal to 
$$
\lim_{b_{1},b_{2} \to 0}{_{2}\phi _1}\left(\begin{matrix} qy/ab_{1},qy/ab_{2} \\ qy/x \end{matrix};q,\frac{b_{1}b_{2}a}{xy}\right)
   =
   {_{0}\phi _1}\left(\begin{matrix} - \\ qy/x \end{matrix};q,\frac{q^{2}y}{ax}\right).
$$
Thus, we have the translation formula (\ref{eq:mu trans 2}). 


Finally, in (\ref{eq:Slater A 2 3}), by taking the limit
$$
a_{1} \to x/a, \quad a_{2} \to y/a, \quad c_{1} \to qx/a, \quad c_{2} \to qy/a
$$
and a similar calculation, we obtain the conclusion (\ref{eq:mu trans 3}). \qed

\begin{rmk}
{\rm{(1)}} Although the formula (\ref{eq:Slater A 2 1}) has the free parameters $x^{\prime }$ and $y^{\prime }$, the formulas (\ref{eq:Slater A 2 2}) and (\ref{eq:Slater A 2 3}) can not be derived by taking the limit $x^{\prime } \to a/x$ or $y^{\prime } \to a/y$ in (\ref{eq:Slater A 2 1}) directly.\\
{\rm{(2)}} The right side of ${_{0}\phi _1}$ in (\ref{eq:mu trans 2}) or (\ref{eq:mu trans 3}) can be regarded as Jackson's $q$-Bessel function:
\begin{align*}
J_\nu^{(2)}(x;q)
   :=
   \frac{(q^{\nu +1})_{\infty}}{(q)_{\infty}}\left(\frac{x}{2}\right)^{\nu }
   {_{0}\phi _1}\left(\begin{matrix} - \\ q^{\nu +1} \end{matrix};q,-\frac{x^{2}q^{\nu +1}}{4}\right)
   =
   \frac{1}{(q)_{\infty}}\left(\frac{x}{2}\right)^{\nu }
   {_{1}\phi _1}\left(\begin{matrix} -x^{2}/4 \\ 0 \end{matrix};q,q^{\nu +1} \right).
\end{align*}
We remark that the $q$-Bessel function term on the right hand side of (\ref{eq:mu trans 2}) and (\ref{eq:mu trans 3}) is not only a fundamental solution of the $q$-Bessel equation, but also the convergent series solution of the $q$-Hermite-Weber equation (3) exactly (see \cite[Lemma~2.2]{ST}).\\
{\rm{(3)}} The formulas (\ref{eq:mu trans 2}) and (\ref{eq:mu trans 3}) are equivalent to (\ref{eq:general mu trans 1}) and (\ref{eq:general mu trans 2}) respectively. 
This fact can be seen from a simple calculation using some formulas for the theta functions:
$$
\theta (x)=\theta (q/x), \quad 
\vartheta _{11}(u)
   =
   -iq^{\frac{1}{8}}e^{-\pi iu}(q)_{\infty }
   \theta (e^{2\pi i u}) 
$$
and the symmetry of the generalized $\mu$-function $\mu (u,v;\alpha )$ \cite[(1.29)]{ST}:
$$
\mu (x,y;a)
   =
   \mu (u,v;\alpha )
   =
   \frac{\vartheta _{11}(v-\alpha \tau )\vartheta _{11}(u)}{\vartheta _{11}(u-\alpha \tau )\vartheta _{11}(v)}e^{2\pi i\alpha (u-v)}
   \mu (v,u;\alpha )
   =
   \frac{\theta (y/a)\theta (x)}{\theta (x/a)\theta (y)}\left(\frac{x}{y}\right)^{\alpha }
   \mu (y,x;a).
$$
\end{rmk}

\noindent
{\bf{Proof of Theorem~\ref{thm:main theorem2}}} 
By taking the limit
$$
a \to x/a, \quad b \to y/a, \quad c \to 0, \quad d \to 0, \quad x \to a
$$
in (\ref{eq:Slater BC 8 2psi2 2}), we obtain
\begin{align}
\label{eq:generalized mu BC0}
{_{2}\psi _2}\left(\begin{matrix} x/a,y/a \\ 0,0 \end{matrix};q,a\right)
   =
   \frac{(x,y)_{\infty }}{(a,xy/a,q^{2}a/xy)_{\infty }}
   {_{4}\psi _8}\left(\begin{matrix} \sqrt{qxy/a},-\sqrt{qxy/a},x/a,y/a \\ \sqrt{xy/qa},-\sqrt{xy/qa},y,x,0,0,0,0 \end{matrix};q,\frac{x^{2}y^{2}}{qa}\right).
\end{align}
The formulas (\ref{eq:generalized mu BC1}) and (\ref{eq:generalized mu BC2}) follow from (\ref{eq:generalized mu BC0}) and (\ref{eq:mu 2psi2}) immediately. \qed

As one interesting application of (\ref{eq:mu BC2}) and (\ref{eq:mu BC3}), we give the following new $q$-expansions of some elliptic functions. 
\begin{cor}
\label{cor: mu elliptic functions}
The Weierstrass $\wp$ function:
$$
\wp (u)=\wp (u,\tau ):=\frac{1}{u^{2}}+\sum_{\substack{\omega \in \mathbb{Z}+\mathbb{Z}\tau, \\ \omega \not=0}}\left\{\frac{1}{(u-\omega )^{2}}-\frac{1}{\omega ^{2}}\right\}
$$
has the following $q$-expansions:
\begin{align}
\wp{(u)}-\wp{(v)}
   &=
   -\frac{4\pi^{2}x}{q}
   \frac{q+x}{q-x}\frac{(q)_{\infty }^{2}}{\theta (x^{2}/q^{2})}{_{2}\psi _6}\left(\begin{matrix} -x,x/q \\ -x/q,x,0,0,0,0 \end{matrix};q,\frac{x^{4}}{q^{2}}\right) \nonumber \\
\label{eq: mu elliptic functions1}
   & \quad +
   \frac{4\pi^{2}y}{q}
   \frac{q+y}{q-y}\frac{(q)_{\infty }^{2}}{\theta (y^{2}/q^{2})}{_{2}\psi _6}\left(\begin{matrix} -y,y/q \\ -y/q,y,0,0,0,0 \end{matrix};q,\frac{y^{4}}{q^{2}}\right) \\   
   &=
\label{eq: mu elliptic functions2}
   -4\pi^{2}(q)_{\infty }^{2}
   \sum_{n \in \mathbb{Z}}
   \left\{\frac{1+xq^{n}}{1-xq^{n}}\frac{x^{4n+1}q^{2n^{2}}}{\theta(x^{2})}
   -\frac{1+yq^{n}}{1-yq^{n}}\frac{y^{4n+1}q^{2n^{2}}}{\theta(y^{2})}\right\} \\
   &=
   -\frac{4\pi ^{2}x(q)_{\infty }^{2}}{\theta(x^{2})}
   \sum_{n \in \mathbb{Z}}
   \left\{\frac{x^{4n}q^{2n^{2}}}{(1-xq^{n})^{2}}-\frac{x^{-4n}q^{2n^{2}}}{(1-x^{-1}q^{n})^{2}}\right\} \nonumber \\
\label{eq: mu elliptic functions3}
   & \quad +
      \frac{4\pi ^{2}y(q)_{\infty }^{2}}{\theta(y^{2})}
   \sum_{n \in \mathbb{Z}}
   \left\{\frac{y^{4n}q^{2n^{2}}}{(1-yq^{n})^{2}}-\frac{y^{-4n}q^{2n^{2}}}{(1-y^{-1}q^{n})^{2}}\right\}.
\end{align}
Further, for some Jacobi elliptic function (see \cite{BW}):
$$
\frac{1}{\sn{(2Ku,k)}\sn{\big(2K\big(u+\frac{1}{2}\big),k\big)}}=\frac{\dn{(2Ku,k)}}{\sn{(2Ku,k)}\cn{(2Ku,k)}},
$$
we have
\begin{align}
& \frac{1}{2\pi{\rm i}}\frac{2K}{\vartheta _{11}\big(\frac{1}{2}\big)}\frac{\dn{(2Ku,k)}}{\sn{(2Ku,k)}\cn{(2Ku,k)}} \nonumber \\
\label{eq: mu elliptic functions4}
   &=
   -q^{-\frac{9}{8}}\frac{q^{2}+x^{2}}{q^{2}-x^{2}}\frac{x}{(q)_{\infty }\theta (-x^{2}/q^{2})}
   {_{4}\psi _8}\left(\begin{matrix} ix,-ix,x/q,-x/q \\ ix/q,-ix/q,x,-x,0,0,0,0 \end{matrix};q,\frac{x^{4}}{q^{2}}\right) \\ 
\label{eq: mu elliptic functions5}
   &=
   q^{-\frac{1}{8}}\frac{x}{(q)_{\infty }\theta(-x^{2})}
   \sum_{n \in \mathbb{Z}}\frac{1+x^{2}q^{2n}}{1-x^{2}q^{2n}}x^{4n}q^{2n^{2}} \\
\label{eq: mu elliptic functions6}
   &=
   q^{-\frac{1}{8}}\frac{x}{(q)_{\infty }\theta(-x^{2})}
   \sum_{n \in \mathbb{Z}}\left\{\frac{x^{4n}q^{2n^{2}}}{1-x^{2}q^{2n}}-\frac{x^{-4n}q^{2n^{2}}}{1-x^{-2}q^{2n}}\right\}.
\end{align}
Here, $k=k(\tau )$ is the elliptic modulus and $K=K(\tau )$ is the elliptic period:
$$
K:=\frac{\pi }{2}(q)_{\infty }^{2}\big({-}q^{\frac{1}{2}}\big)_{\infty }^{4}=\int_{0}^{1}\frac{{\rm d}t}{\sqrt{\big(1-t^{2}\big)\big(1-k^{2}t^{2}\big)}}.
$$
\end{cor}
\begin{proof}
The formulas (\ref{eq: mu elliptic functions4}), (\ref{eq: mu elliptic functions5}), and (\ref{eq: mu elliptic functions6}) follow immediately from (\ref{eq:mu BC1}), (\ref{eq:mu BC2}), and  \cite[(3.9)]{ST}:
$$
\mu{\left(u,u+\frac{1}{2}\right)}
   =
   -\frac{1}{2\pi{\rm i}}\frac{2K}{\vartheta _{11}\big(\frac{1}{2}\big)}\frac{\dn{(2Ku,k)}}{\sn{(2Ku,k)}\cn{(2Ku,k)}}.
$$

To prove (\ref{eq: mu elliptic functions1}), (\ref{eq: mu elliptic functions2}), and (\ref{eq: mu elliptic functions3}), it suffices to prove the following fact:
\begin{align}
\label{eq:W-M}
\wp{(u)}-\wp{(v)}
   =
   M(u)-M(v), \quad M(u):=4\pi^{2}iq^{\frac{1}{8}}(q)_{\infty }^{3}\mu(u,u), \quad u,v \in \mathbb{C}\setminus \mathbb{Z}+\mathbb{Z}\tau .
\end{align}
We recall some fundamental properties of the original $\mu$-function (see \cite[Proposition~1.4]{Zw}, \cite[Proposition~8.2, 8.3]{BFOR}), that is, the symmetry: $\mu(u,v)=\mu(-u,-v)=\mu(v,u)$, and periodicity: $\mu(u+1,v)=\mu(u,v+1)=-\mu(u,v)$, $\mu(u+\tau, v+\tau )=\mu(u,v)$. 
Then the double periodicity of $M(u)$ holds:
$$
M(u+m+n\tau )=M(u) \quad m,n \in \mathbb{Z}.
$$
By the definitions and the even symmetry for $\wp{(u)}$ and $M(u)$, we see that $\wp{(u)}-M(u)$ is an even entire function. 
Hence, $\wp{(u)}-M(u)$ is a constant independent of $u$ and we obtain the above fact (\ref{eq:W-M}). 
\end{proof}
\begin{rmk}
The left hand side of (\ref{eq: mu elliptic functions1}) is written by 
\begin{align}
\label{eq:wp and sigma}
\wp{(u)}-\wp{(v)}
   =
   -\frac{\sigma (u+v)\sigma (u-v)}{\sigma (u)^{2}\sigma (v)^{2}}
\end{align}
where $\sigma$ is the Weierstrass $\sigma$ function:
$$
\sigma (u)
   :=
   u
   \prod_{\omega \not=0 \in \mathbb{Z}+\mathbb{Z}\tau }
   \left(1-\frac{u}{\omega }\right)e^{\frac{u}{\omega }+\frac{u^{2}}{2\omega ^{2}}}.
$$

Bailey \cite{B2} noted that the classical well-known formula (\ref{eq:wp and sigma}) is also derived from the specialization of Bailey's summation (\ref{eq:Bailey's summation}):
\begin{align}
\wp(u)-\wp(v)
   &=
   -4\pi^{2}\frac{(1-xy)(x-y)}{(1-x)^{2}(1-y)^{2}}
   {_{6}\psi _6}\left(\begin{matrix} q\sqrt{xy},-q\sqrt{xy},x,x,y,y \\ \sqrt{xy},-\sqrt{xy},qy,qy,qx,qx \end{matrix};q,q\right) \nonumber \\
   &=
   -4\pi^{2}
   \sum_{n \in \mathbb{Z}}\frac{(x-y)(1-xyq^{2n})q^{n}}{(1-xq^{n})^{2}(1-yq^{n})^{2}} \\
   &=
   -4\pi^{2}
   \sum_{n \in \mathbb{Z}}\left\{\frac{xq^{n}}{(1-xq^{n})^{2}}-\frac{yq^{n}}{(1-yq^{n})^{2}}\right\}.
\end{align}
By comparing our results and Bailey's calculation, we have the following interesting relation for bilateral basic hypergeometric functions:
\begin{align}
& \frac{(1-xy)(x-y)}{(1-x)^{2}(1-y)^{2}}
   {_{6}\psi _6}\left(\begin{matrix} q\sqrt{xy},-q\sqrt{xy},x,x,y,y \\ \sqrt{xy},-\sqrt{xy},qy,qy,qx,qx \end{matrix};q,q\right) \nonumber \\
   &=
   \frac{x}{q}
   \frac{q+x}{q-x}\frac{(q)_{\infty }^{2}}{\theta (x^{2}/q^{2})}{_{2}\psi _6}\left(\begin{matrix} -x,x/q \\ -x/q,x,0,0,0,0 \end{matrix};q,\frac{x^{2}}{q^{2}}\right) \nonumber \\
\label{eq:curious psi rel}
   & \quad -
   \frac{y}{q}
   \frac{q+y}{q-y}\frac{(q)_{\infty }^{2}}{\theta (y^{2}/q^{2})}{_{2}\psi _6}\left(\begin{matrix} -y,y/q \\ -y/q,y,0,0,0,0 \end{matrix};q,\frac{y^{2}}{q^{2}}\right).
\end{align}

Can we derive this curious relation (\ref{eq:curious psi rel}) directly by taking some appropriate limit of Slater's transformation formula (\ref{eq:Slater BC r})? 
\end{rmk}


\appendix
\section{Degenerate Very-Well-Poised bilateral series expressions for Ramanujan's mock theta functions}
We define the theta function $\theta_{q}(x)$ and the degenerate Very-Well-Poised bilateral $q$-hypergeometric series $\mathcal{W}(a;b,c;q)$ by 
\begin{align}
\theta_{q}(x)
   :&=
   (q,-x,-q/x)_{\infty }
   =
   \sum_{n \in \Zz}
   x^{n}q^{\frac{n(n-1)}{2}}, \nonumber \\
\mathcal{W}(a;b,c;q)
   :&=
   \frac{1-a}{(1-a/b)(1-a/c)}
   {_{4}\psi _8}\left( \begin{matrix} \sqrt{a}q,-\sqrt{a}q,b,c \\ \sqrt{a},-\sqrt{a},aq/b,aq/c,0,0,0,0 \end{matrix};q,\frac{a^3q^2}{bc} \right) \\
\label{degenerated VWP}
    &=\sum_{n\in\Zz}\frac{(1-a q^{2n})(b,c;q)_n}{(a/b,a/c)_{n+1}}q^{2n^2}\left(\frac{a^3}{bc}\right)^n.
\end{align}
The relation between $\mu{(x,y;a,q)}$ and $\tW$, or $\mu{(x,y;q)}$ and $\tW$, is as follow:
\begin{align}
\mu{(x,y;a,q)}
   &=
   -iq^{-\frac{1}{8}}
   \left(\frac{x}{y}\right)^{\frac{\alpha}{2}}
   \frac{(x/q,y/q,aq/x,aq/y)_{\infty }}{(q)_{\infty }\theta(y)\theta(x/a)\theta(xy/aq)}
   \mathcal{W}\left(\frac{xy}{aq};\frac{x}{a},\frac{y}{a};q\right), \\
\mu{(x,y;q)}
   &=
   iq^{-\frac{9}{8}}
   \frac{\sqrt{xy}}{(q)_{\infty }\theta(xy/q^{2})}
   \mathcal{W}\left(\frac{xy}{q^{2}};\frac{x}{q},\frac{y}{q};q\right)
   =
   iq^{-\frac{9}{8}}
   \frac{\sqrt{xy}}{\theta_{q}(-xy/q^{2})}
   \mathcal{W}\left(\frac{xy}{q^{2}};\frac{x}{q},\frac{y}{q};q\right).   
\end{align}

In this appendix, we list the expressions by $\tW$ for all of Ramanujan's mock theta functions (see \cite[Appendix~A]{BFOR}, but this reference has an incorrect definition of the third order mock theta function $\rho(q)$). 
\subsection{Order 2 theta functions}

\begin{align}
A(q)
   :&=
   \sum_{n=0}^{\infty}
      \frac{q^{n+1}(-q^2;q^2)_n}{(q;q^2)_{n+1}} \nonumber \\
   &=
   \frac{q^2}{\theta_{q^4}(-q^5)}\tW(q^5;q^3,q^2;q^4) \nonumber \\
   &=
   \frac{1}{\theta_{q^4}(-q^5)}
   \sum_{n\in\Zz}
      \frac{(1-q^{8n+5})q^{8n^2+10n+2}}{(1-q^{4n+3})(1-q^{4n+2})}, \\
B(q)
   :&=
   \sum_{n=0}^{\infty}
      \frac{q^n(-q;q^2)_n}{(q;q^2)_{n+1}} \nonumber \\
   &=
   \frac{q^2}{\theta_{q^4}(-q^6)}\tW(q^6;q^3,q^3;q^4) \nonumber \\
   &=
   \frac{1}{\theta_{q^4}(-q^6)}
   \sum_{n\in\Zz}
      \frac{1+q^{4n+3}}{1-q^{4n+3}}q^{8n^2+12n+2}, \\
\mu(q)
   :&=
   \sum_{n=0}^{\infty}
      \frac{(-1)^nq^{n^2}(q;q^2)_n}{(-q^2,q^2)_n^2} \nonumber \\
   &=
   \frac{4}{\theta_{q^4}(q)}\tW(-q;q,-1;q^4)-\frac{(q^2;q^2)_\infty^8}{(q;q)_\infty^3(q^4;q^4)_\infty^4} \nonumber \\
   &=
   \frac{4}{\theta_{q^4}(q)}
   \sum_{n\in\Zz}
      \frac{(1+q^{8n+1})q^{8n^2+2n}}{(1-q^{4n+1})(1+q^{4n})}
   -\frac{(q^2;q^2)_\infty^8}{(q;q)_\infty^3(q^4;q^4)_\infty^4}.
\end{align}
\subsection{Order 3 mock theta functions}

\begin{align}
f(q)
   :&=
   \sum_{n=0}^{\infty}
      \frac{q^{n^2}}{(-q;q)_n^2} \nonumber \\
   &=
   -\frac{4q}{\theta_{q^3}(q^3)}\tW(-q^3;-q^2,q;q^3)
   +\frac{(q^3;q^3)_\infty^4}{(q;q)_\infty(q^6;q^6)_\infty^2} \nonumber \\
   &=
   -\frac{4}{\theta_{q^3}(q^3)}
   \sum_{n\in\Zz}
      \frac{(1+q^{6n+3})q^{6n^2+6n+1}}{(1-q^{3n+1})(1+q^{3n+2})}
   +\frac{(q^3;q^3)_\infty^4}{(q;q)_\infty(q^6;q^6)_\infty^2}, \\
\phi(q)
   :&=
   \sum_{n=0}^{\infty}
      \frac{q^{n^2}}{(-q^2;q^2)_n} \nonumber \\
   &=
   -\frac{2q^2}{\theta_{-q^3}(-q^5)}\tW(q^5;q^3,q^2;-q^3)
   +2q\frac{(q^6;q^6)_\infty(q^{12};q^{12})_\infty^2}{(q^3;q^3)_\infty(q^4;q^4)_\infty} \nonumber \\
   &=
   -\frac{2}{\theta_{-q^3}(-q^5)}
   \sum_{n\in\Zz}
      \frac{(1-q^{6n+5})q^{6n^2+10n+2}}{(1-(-1)^nq^{3n+2})(1-(-1)^nq^{3n+3})}
   +2q\frac{(q^6;q^6)_\infty(q^{12};q^{12})_\infty^2}{(q^3;q^3)_\infty(q^4;q^4)_\infty}, \\
\psi(q)
   :&=
   \sum_{n=1}^{\infty}
      \frac{q^{n^2}}{(q;q^2)_n} \nonumber \\
   &=
   \frac{q}{\theta_{-q^3}(-q^3)}
   \tW(q^3;q^2,q;-q^3)
   +q\frac{(q^6;q^6)_\infty(q^{12};q^{12})_\infty^2}{(q^3;q^3)_\infty(q^4;q^4)_\infty} \nonumber \\
   &=
   \frac{1}{\theta_{-q^3}(-q^3)}
   \sum_{n\in\Zz}
      \frac{(1-q^{6n+3})q^{6n^2+6n+1}}{(1-(-1)^nq^{3n+1})(1-(-1)^nq^{3n+2})}
   +q\frac{(q^6;q^6)_\infty(q^{12};q^{12})_\infty^2}{(q^3;q^3)_\infty(q^4;q^4)_\infty}, \\
\chi(q)
   :&=
   \sum_{n=0}^{\infty}
      \frac{q^{n^2}(-q;q)_n}{(-q^3;q^3)_n} \nonumber \\
   &=
   -\frac{q}{\theta_{q^3}(q^3)}\tW(-q^3;-q^2,q;q^3)
   +\frac{(q^3;q^3)_\infty^4}{(q;q)_\infty(q^6;q^6)_\infty^2} \nonumber \\
   &=
   -\frac{1}{\theta_{q^3}(q^3)}
   \sum_{n\in\Zz}
      \frac{(1+q^{6n+3})q^{6n^2+6n+1}}{(1-q^{3n+1})(1+q^{3n+2})}
   +\frac{(q^3;q^3)_\infty^4}{(q;q)_\infty(q^6;q^6)_\infty^2}, \\
\omega(q)
   :&=
   \sum_{n=0}^{\infty}
      \frac{q^{2n(n+1)}}{(q;q^2)_{n+1}^2} \nonumber \\
   &=
   \frac{2q}{\theta_{q^6}(-q^5)}\tW(q^5;q^3,q^2;q^6)
   +\frac{(q^6;q^6)_\infty^4}{(q^2;q^2)_\infty(q^3;q^3)_\infty^2} \nonumber \\
   &=
   \frac{2}{\theta_{q^6}(-q^5)}
   \sum_{n\in\Zz}
      \frac{(1-q^{12n+5})q^{12n^2+10n+1}}{(1-q^{6n+2})(1-q^{6n+3})}
   +\frac{(q^6;q^6)_\infty^4}{(q^2;q^2)_\infty(q^3;q^3)_\infty^2}, \\
\nu(q)
   :&=
   \sum_{n=0}^{\infty}
   \frac{q^{n(n+1)}}{(-q;q^2)_{n+1}} \nonumber \\
   &=
   \frac{2q^2}{\theta_{q^{12}}(-q^8)}\tW(q^8;-q^5,-q^3;q^{12})
   +\frac{(q;q)_\infty(q^3;q^3)_\infty(q^{12};q^{12})_\infty}{(q^2;q^2)_\infty(q^6;q^6)_\infty} \nonumber \\
   &=
   \frac{2}{\theta_{q^{12}}(-q^8)}
   \sum_{n\in\Zz}
      \frac{(1-q^{24n+8})q^{24n^2+16n+2}}{(1+q^{12n+5})(1+q^{12n+3})}
   +\frac{(q;q)_\infty(q^3;q^3)_\infty(q^{12};q^{12})_\infty}{(q^2;q^2)_\infty(q^6;q^6)_\infty}, \\
\rho(q)
   :&=
   \sum_{n=1}^{\infty}
      \frac{q^{2n(n-1)}(q;q^2)_n}{(q^3;q^6)_n} \nonumber \\
   &=
   \frac{1}{\theta_{q^6}(-q^3)}\tW(q^3;-q^2,-q;q^6) \nonumber \\
   &=
   \frac{1}{\theta_{q^6}(-q^3)}
   \sum_{n\in\Zz}
      \frac{(1-q^{12n+3})q^{12n^2+6n}}{(1+q^{6n+1})(1+q^{6n+2})}.
\end{align}

\subsection{Order 5 mock theta functions}
\begin{align}
f_0(q)
   :&=
   \sum_{n=0}^{\infty}
      \frac{q^{n^2}}{(-q;q)_n} \nonumber \\
   &=
   -\frac{2q^4}{\theta_{q^{30}}(-q^{22})}\tW(q^{22};q^{18},q^4;q^{30}) \nonumber \\
   & \quad 
   -\frac{2q^2}{\theta_{q^{30}}(-q^{12})}\tW(q^{12};q^8,q^4;q^{30})
   +\frac{(q^5;q^5)_\infty^3}{\theta_{q^5}(-q)(q^{10};q^{10})_\infty} \nonumber \\
   &=
   -\frac{2}{\theta_{q^{30}}(-q^{22})}
   \sum_{n\in\Zz}
      \frac{(1-q^{60n+22})q^{60n^2+44n+4}}{(1-q^{30n+18})(1-q^{30n+4})} \nonumber \\
   & \quad
   -\frac{2}{\theta_{q^{30}}(-q^{12})}
   \sum_{n\in\Zz}
      \frac{(1-q^{60n+12})q^{60n^2+24n+2}}{(1-q^{30n+8})(1-q^{30n+4})}
   +\frac{(q^5;q^5)_\infty^3}{\theta_{q^5}(-q)(q^{10};q^{10})_\infty}, \\
f_1(q)
   :&=
   \sum_{n=0}^{\infty}
      \frac{q^{n(n+1)}}{(-q;q)_n} \nonumber \\
   &=
   -\frac{2q^7}{\theta_{q^{30}}(-q^{24})}\tW(q^{24};q^{16},q^8;q^{30}) \nonumber \\
   & \quad 
   -\frac{2q^3}{\theta_{q^{30}}(-q^{14})}\tW(q^{14};q^8,q^6;q^{30})
   +\frac{(q^5;q^5)_\infty^3}{\theta_{q^5}(-q^2)(q^{10};q^{10})_\infty} \nonumber \\
   &=
   -\frac{2}{\theta_{q^{30}}(-q^{24})}
   \sum_{n\in\Zz}
      \frac{(1-q^{60n+24})q^{60n^2+48n+7}}{(1-q^{30n+16})(1-q^{30n+8})} \nonumber \\
   & \quad
   -\frac{2q^3}{\theta_{q^{30}}(-q^{14})}
   \sum_{n\in\Zz}
      \frac{(1-q^{60n+14})q^{60n^2+28n+3}}{(1-q^{30n+8})(1-q^{30n+6})}
   +\frac{(q^5;q^5)_\infty^3}{\theta_{q^5}(-q^2)(q^{10};q^{10})_\infty}, \\
F_0(q)
   :&=
   \sum_{n=0}^{\infty}
      \frac{q^{2n^2}}{(q;q^2)_n} \nonumber \\
   &=
   \frac{1}{\theta_{q^{15}}(-q^4)}\tW(q^4;q^3,q;q^{15}) \nonumber \\
   & \quad 
   -\frac{q^4}{\theta_{q^{15}}(-q^{16})}\tW(q^{16};q^{12},q^4;q^{15})
   -\frac{q(q^{10};q^{10})_\infty^3}{(q^5;q^5)_\infty\theta_{q^{10}}(-q^4)} \nonumber \\
   &=
   \frac{1}{\theta_{q^{15}}(-q^4)}
   \sum_{n\in\Zz}
      \frac{(1-q^{30n+4})q^{30n^2+8n}}{(1-q^{15n+3})(1-q^{15n+1})} \nonumber \\
   & \quad
   -\frac{1}{\theta_{q^{15}}(-q^{16})}
   \sum_{n\in\Zz}
      \frac{(1-q^{30n+16})q^{30n^2+32n+4}}{(1-q^{15n+12})(1-q^{15n+4})}
   -\frac{q(q^{10};q^{10})_\infty^3}{(q^5;q^5)_\infty\theta_{q^{10}}(-q^4)}, \\
F_1(q)
   :&=
   \sum_{n=0}^{\infty}
      \frac{q^{2n(n+1)}}{(q;q^2)_{n+1}} \nonumber \\
   &=
   \frac{q}{\theta_{q^{15}}(-q^7)}\tW(q^7;q^4,q^3;q^{15}) \nonumber \\ 
   & \quad 
   +\frac{q^3}{\theta_{q^{15}}(-q^{12})}\tW(q^{12};q^8,q^4;q^{15})
   +\frac{(q^{10};q^{10})_\infty^3}{(q^5;q^5)_\infty\theta_{q^{10}}(-q^2)} \nonumber \\
   &=
   \frac{1}{\theta_{q^{15}}(-q^7)}
   \sum_{n\in\Zz}
      \frac{(1-q^{30n+7})q^{30n^2+14n+1}}{(1-q^{15n+4})(1-q^{15n+3})} \nonumber \\
   & \quad
   +\frac{1}{\theta_{q^{15}}(-q^{12})}
   \sum_{n\in\Zz}
      \frac{(1-q^{30n+12})q^{30n^2+24n+3}}{(1-q^{15n+8})(1-q^{15n+4})}
   +\frac{(q^{10};q^{10})_\infty^3}{(q^5;q^5)_\infty\theta_{q^{10}}(-q^2)}, \\
\phi_0(q)
   :&=
   \sum_{n=0}^{\infty}
      q^{n^2}(-q;q^2)_n \nonumber \\
   &=
   \frac{q^8}{\theta_{-q^{15}}(-q^{20})}\tW(q^{20};q^{11},q^9;-q^{15})
   -\frac{q^9}{\theta_{-q^{15}}(q^{25})}\tW(-q^{25};-q^{16},q^9;-q^{15}) \nonumber \\
   &=
   \frac{1}{\theta_{-q^{15}}(-q^{20})}
   \sum_{n\in\Zz}
      \frac{(1-q^{30n+20})q^{30n^2+40n+8}}{(1-(-1)^nq^{15n+11})(1-(-1)^nq^{15n+9})} \nonumber \\
   & \quad
   -\frac{1}{\theta_{-q^{15}}(q^{25})}
   \sum_{n\in\Zz}
      \frac{(1+q^{30n+25})q^{30n^2+50n+9}}{(1+(-1)^nq^{15n+16})(1-(-1)^nq^{15n+9})}, \\
\phi_1(q)
   :&=
   \sum_{n=0}^{\infty}
      q^{(n+1)^2}(-q;q^2)_n \nonumber \\
   &=
   \frac{q^3}{\theta_{-q^{15}}(-q^{10})}\tW(q^{10};q^7,q^3;-q^{15})
   +\frac{q}{\theta_{-q^{15}}(q^5)}\tW(-q^5;q^3,-q^2;-q^{15}) \nonumber \\
   &=
   \frac{1}{\theta_{-q^{15}}(-q^{10})}
   \sum_{n\in\Zz}
      \frac{(1-q^{30n+10})q^{30n^2+20n+3}}{(1-(-1)^nq^{15n+7})(1-(-1)^nq^{15n+3})} \nonumber \\
   & \quad
   +\frac{1}{\theta_{-q^{15}}(q^5)}\sum_{n\in\Zz}\frac{(1+q^{30n+5})q^{30n^2+10n+1}}{(1-(-1)^nq^{15n+3})(1+(-1)^nq^{15n+2})}, \\
\psi_0(q)
   :&=
   \sum_{n=0}^{\infty}
   q^{\frac{(n+1)(n+2)}{2}}(-q;q)_n \nonumber \\
   &=
   \frac{q^3}{\theta_{q^{30}}(-q^{20})}\tW(q^{20};q^{17},q^3;q^{30})
   +\frac{q}{\theta_{q^{30}}(-q^{10})}\tW(q^{10};q^7,q^3;q^{30}) \nonumber \\
   &=
   \frac{1}{\theta_{q^{30}}(-q^{20})}
   \sum_{n\in\Zz}
      \frac{(1-q^{60n+20})q^{60n^2+40n+3}}{(1-q^{30n+17})(1-q^{30n+3})} \nonumber \\
   & \quad
   +\frac{1}{\theta_{q^{30}}(-q^{10})}
   \sum_{n\in\Zz}
      \frac{(1-q^{60n+10})q^{60n^2+20n+1}}{(1-q^{30n+7})(1-q^{30n+3})}, \\
\psi_1(q)
   :&=
   \sum_{n=0}^{\infty}
      q^\frac{n(n+1)}{2}(-q;q)_n \nonumber \\
   &=
   \frac{1}{\theta_{q^{30}}(-q^{10})}\tW(q^{10};q^9,q;q^{30})
   +\frac{q^6}{\theta_{q^{30}}(-q^{20})}\tW(q^{20};q^{11},q^9;q^{30}) \nonumber \\
   &=
   \frac{1}{\theta_{q^{30}}(-q^{10})}
   \sum_{n\in\Zz}
      \frac{(1-q^{60n+10})q^{60n^2+20n}}{(1-q^{30n+9})(1-q^{30n+1})} \nonumber \\
   & \quad
   +\frac{1}{\theta_{q^{30}}(-q^{20})}
   \sum_{n\in\Zz}
      \frac{(1-q^{60n+20})q^{60n^2+40n+6}}{(1-q^{30n+11})(1-q^{30n+9})}, \\
\chi_0(q)
   :&=
   \sum_{n=0}^{\infty}
      \frac{q^n}{(q^{n+1};q)_n} \nonumber \\
   &=
   2+\frac{3 q^{-6}}{\theta_{q^{15}}(-q^{-5})}\tW(q^{-5};q,q^{-6};q^{15}) \nonumber \\
   & \quad 
   +\frac{3 q^3}{\theta_{q^{15}}(-q^{10})}\tW(q^{10};q^6,q^4;q^{15})
   +2\frac{\theta_{q^5}(-q^2)^3}{(q;q)_\infty^2} \nonumber \\
   &=
   2
   +\frac{3}{\theta_{q^{15}}(-q^{-5})}
   \sum_{n\in\Zz}
      \frac{(1-q^{30n-5})q^{30n^2-10n-6}}{(1-q^{30n+1})(1-q^{30n-6})} \nonumber \\
   & \quad 
   +\frac{3}{\theta_{q^{15}}(-q^{10})}
   \sum_{n\in\Zz}
      \frac{(1-q^{30n+10})q^{30n^2+20n+3}}{(1-q^{15n+6})(1-q^{15n+4})}
   +2\frac{\theta_{q^5}(-q^2)^3}{(q;q)_\infty^2}, \\
\chi_1(q)
   :&=
   \sum_{n=0}^{\infty}
      \frac{q^{n}}{(q^{n+1};q)_{n+1}} \nonumber \\
   &=
   \frac{3q^2}{\theta_{q^{15}}(-q^{10})}\tW(q^{10};q^7,q^3;q^{15}) \nonumber \\
   & \quad 
   +\frac{3}{\theta_{q^{15}}(-q^5)}\tW(q^5;q^3,q^2;q^{15})
   -2\frac{\theta_{q^5}(-q)^3}{(q;q)_\infty^2} \nonumber \\
   &=\frac{3}{\theta_{q^{15}}(-q^{10})}
   \sum_{n\in\Zz}
      \frac{(1-q^{30n+10})q^{30n^2+20n+2}}{(1-q^{15n+7})(1-q^{15n+3})} \nonumber \\
   & \quad 
   +\frac{3}{\theta_{q^{15}}(-q^5)}
   \sum_{n\in\Zz}
      \frac{(1-q^{30n+5})q^{30n^2+10n}}{(1-q^{15n+3})(1-q^{15n+2})}
   -2\frac{\theta_{q^5}(-q)^3}{(q;q)_\infty^2}, \\
\Psi_0(q)
   :&=
   -1
   +\sum_{n=0}^{\infty}
      \frac{q^{5n^2}}{(q;q^5)_{n+1}(q^4;q^5)_n} \nonumber \\
   &=
   \frac{q}{\theta_{q^{15}}(-q^6)}\tW(q^6;q^4,q^2;q^{15})
   +\frac{q^2}{\theta_{q^{15}}(-q^{11})}\tW(q^{11};q^9,q^2;q^{15}) \nonumber \\
   &=
   \frac{1}{\theta_{q^{15}}(-q^6)}
   \sum_{n\in\Zz}
      \frac{(1-q^{30n+6})q^{30n^2+12n+1}}{(1-q^{15n+4})(1-q^{15n+2})} \nonumber \\
   & \quad
   +\frac{1}{\theta_{q^{15}}(-q^{11})}
   \sum_{n\in\Zz}
      \frac{(1-q^{30n+11})q^{30n^2+22n+2}}{(1-q^{15n+9})(1-q^{15n+2})}, \\
\Psi_1(q)
   :&=
   -1
   +\sum_{n=0}^{\infty}
      \frac{q^{5n^2}}{(q^2;q^5)_{n+1}(q^3;q^5)_n} \nonumber \\
   &=
   \frac{q^2}{\theta_{q^{15}}(-q^7)}\tW(q^7;q^4,q^3;q^{15})
   +\frac{q^4}{\theta_{q^{15}}(-q^{12})}\tW(q^{12};q^8,q^4;q^{15}) \nonumber \\
   &=
   \frac{1}{\theta_{q^{15}}(-q^7)}
   \sum_{n\in\Zz}
      \frac{(1-q^{30n+7})q^{30n^2+14n+2}}{(1-q^{15n+4})(1-q^{15n+3})} \nonumber \\
   & \quad
   +\frac{1}{\theta_{q^{15}}(-q^{12})}
   \sum_{n\in\Zz}
      \frac{(1-q^{30n+12})q^{30n^2+24n+4}}{(1-q^{15n+8})(1-q^{15n+4})}.
\end{align}

\subsection{Order 6 mock theta functions}

\begin{align}
\phi(q)
   :&=
   \sum_{n=0}^{\infty}
      \frac{(-1)^nq^{n^2}(q;q^2)_n}{(-q;q)_{2n}} \nonumber \\
   &=
   \frac{2}{\theta_{q^3}(-q)}\tW(q;-q,-1;q^3) \nonumber \\
   &=
   \frac{2}{\theta_{q^3}(-q)}
   \sum_{n\in\Zz}
      \frac{(1-q^{6n+1})q^{6n^2+2n}}{(1+q^{3n+1})(1+q^{3n})}, \\
\psi(q)
   :&=
   \sum_{n=0}^{\infty}
      \frac{(-1)^nq^{(n+1)^2}(q;q^2)_n}{(-q;q)_{2n+1}} \nonumber \\
   &=
   \frac{q}{\theta_{q^3}(-q^2)}\tW(q^2;-q,-q;q^3) \nonumber \\
   &=
   \frac{1}{\theta_{q^3}(-q^2)}
   \sum_{n\in\Zz}\frac{1-q^{3n+1}}{1+q^{3n+1}}q^{6n^2+4n+1}, \\
\rho(q)
   :&=
   \sum_{n=0}^{\infty}
      \frac{q^\frac{n(n+1)}{2}(-q;q)_n}{(q;q^2)_{n+1}} \nonumber \\
   &=
   \frac{1}{\theta_{q^6}(-q^2)}\tW(q^2;q,q;q^6) \nonumber \\
   &=
   \frac{1}{\theta_{q^6}(-q^2)}
   \sum_{n\in\Zz}
      \frac{1+q^{6n+1}}{1-q^{6n+1}}q^{12n^2+4n}, \\
\sigma(q)
   :&=
   \sum_{n=0}^{\infty}
      \frac{q^\frac{(n+1)(n+2)}{2}(-q;q)_n}{(q;q^2)_{n+1}} \nonumber \\
   &=
   \frac{q}{\theta_{q^6}(-q^4)}\tW(q^4;q^3,q;q^6) \nonumber \\
   &=
   \frac{1}{\theta_{q^6}(-q^4)}
   \sum_{n\in\Zz}
      \frac{(1-q^{12n+4})q^{12n^2+8n+1}}{(1-q^{6n+3})(1-q^{6n+1})}, \\
\lambda(q)
   :&=
   \sum_{n=0}^{\infty}
      \frac{(-1)^nq^n(q;q^2)_n}{(-q;q)_n} \nonumber \\
   &=
   \frac{2q}{\theta_{q^6}(-q^4)}\tW(q^4;-q^2,-q^2;q^6)
   +\frac{(q;q)_\infty^3(q^6;q^6)_\infty^2}{(q^2;q^2)_\infty^3(q^3;q^3)_\infty} \nonumber \\
   &=
   \frac{2}{\theta_{q^6}(-q^4)}
   \sum_{n\in\Zz}
      \frac{1-q^{6n+2}}{1+q^{6n+2}}q^{12n^2+8n+1}
   +\frac{(q;q)_\infty^3(q^6;q^6)_\infty^2}{(q^2;q^2)_\infty^3(q^3;q^3)_\infty}, \\
\mu(q)
   :&=
   \frac{1}{2}
   +\frac{1}{2}
   \sum_{n=0}^{\infty}
      \frac{(-1)^nq^{n+1}(1+q^n)(q;q^2)_n}{(-q;q)_{n+1}} \nonumber \\
   &=
   \frac{2}{\theta_{q^6}(-q^2)}\tW(q^2;-q^2,-1;q^6)
   -\frac{1}{2}\frac{(q;q)_\infty^2(q^3;q^3)_\infty^2}{(q^2;q^2)_\infty^2(q^6;q^6)_\infty} \nonumber \\
   &=
   \frac{2}{\theta_{q^6}(-q^2)}
   \sum_{n\in\Zz}
      \frac{(1-q^{12n+2})q^{12n^2+4n}}{(1+q^{6n+2})(1+q^{6n})}
   -\frac{1}{2}\frac{(q;q)_\infty^2(q^3;q^3)_\infty^2}{(q^2;q^2)_\infty^2(q^6;q^6)_\infty}, \\
\gamma(q)
   :&=
   \sum_{n=0}^{\infty}
      \frac{q^{n^2}(q;q)_n}{(q^3;q^3)_n} \nonumber \\
   &=
   \frac{3}{\theta_{q^3}(-q)}\tW(q,;-q,-1;q^3)
   -\frac{\theta_{q^2}(-q)^2}{2\theta_{q^3}(q)} \nonumber \\
   &=
   \frac{3}{\theta_{q^3}(-q)}
   \sum_{n\in\Zz}
      \frac{(1-q^{6n+1})q^{6n^2+2n}}{(1+q^{3n+1})(1+q^{3n})}
   -\frac{\theta_{q^2}(-q)^2}{2\theta_{q^3}(q)}, \\
\phi_-(q)
   :&=
   \sum_{n=1}^{\infty}
      \frac{q^n(-q;q)_{2n-1}}{(q;q^2)_n} \nonumber \\
   &=
   -\frac{q^\frac{1}{2}}{\theta_{q^3}(-q^2)}\tW(q^2;-q^\frac{3}{2},-q^\frac{1}{2};q^3)
   +\frac{q^\frac{1}{2}(q^2;2^2)_\infty^2(q^6;q^6)_\infty^2}{(q;q)_\infty^2(q^3;q^3)_\infty} \nonumber \\
   &=
   -\frac{1}{\theta_{q^3}(-q^2)}
   \sum_{n\in\Zz}
      \frac{(1-q^{6n+2})q^{6n^2+4n+\frac{1}{2}}}{(1+q^{3n+\frac{3}{2}})(1+q^{3n+\frac{1}{2}})}
   +\frac{q^\frac{1}{2}(q^2;2^2)_\infty^2(q^6;q^6)_\infty^2}{(q;q)_\infty^2(q^3;q^3)_\infty}, \\
\psi_-(q)
   :&=
   \sum_{n=1}^{\infty}
      \frac{q^n(-q;q)_{2n-2}}{(q;q^2)_n} \nonumber \\
   &=
   \frac{1}{2}\frac{q}{\theta_{q^3}(-q^2)}\tW(q^2;q,q;q^3)
   +\frac{q}{2}\frac{(q^6;q^6)_\infty^3}{(q;q)_\infty(q^2;q^2)_\infty} \nonumber \\
   &=
   \frac{1}{2\theta_{q^3}(-q^2)}
   \sum_{n\in\Zz}
      \frac{1+q^{3n+1}}{1-q^{3n+1}}q^{6n^2+4n+1}
   +\frac{q}{2}\frac{(q^6;q^6)_\infty^3}{(q;q)_\infty(q^2;q^2)_\infty}.
\end{align}

\subsection{Order 7 mock theta functions}

\begin{align}
\mathcal{F}_0(q)
   :&=
   \sum_{n=0}^{\infty}
      \frac{q^{n^2}}{(q^{n+1};q)_n} \nonumber \\
   &=
   \frac{2q^4}{\theta_{q^{21}}(-q^{14})}\tW(q^{14};q^9,q^5;q^{21}) \nonumber \\
   & \quad 
   -\frac{2q^9}{\theta_{q^{21}}(-q^{28})}\tW(q^{28};q^{19},q^9;q^{21})
   +\frac{(q^3,q^4,q^7;q^7)_\infty}{(q,q^2,q^5,q^6;q^7)_\infty} \nonumber \\
   &=
   \frac{2}{\theta_{q^{21}}(-q^{14})}
   \sum_{n\in\Zz}
      \frac{(1-q^{42n+14})q^{42n^2+28n+4}}{(1-q^{21n+9})(1-q^{21n+5})} \nonumber \\
   &\quad
   -\frac{2}{\theta_{q^{21}}(-q^{28})}
   \sum_{n\in\Zz}
      \frac{(1-q^{42n+28})q^{42n^2+56n+9}}{(1-q^{21n+19})(1-q^{21n+9})}
   +\frac{(q^3,q^4,q^7;q^7)_\infty}{(q,q^2,q^5,q^6;q^7)_\infty}, \\
\mathcal{F}_1(q)
   :&=
   \sum_{n=1}^{\infty}
      \frac{q^{n^2}}{(q^n;q)_n} \nonumber \\
   &=
   \frac{2q^3}{\theta_{q^{21}}(-q^{14})}\tW(q^{14};q^{11},q^3;q^{21}) \nonumber \\
   & \quad 
   +\frac{2q}{\theta_{q^{21}}(-q^7)}\tW(q^7;q^4,q^3;q^{21})
   -\frac{q(q,q^6,q^7;q^7)_\infty}{(q^2,q^3,q^4,q^5;q^7)_\infty} \nonumber \\
   &=
   \frac{2}{\theta_{q^{21}}(-q^{14})}
   \sum_{n\in\Zz}
      \frac{(1-q^{42n+14})q^{42n^2+28n+3}}{(1-q^{21n+11})(1-q^{21n+3})} \nonumber \\
   &\quad
   +\frac{2}{\theta_{q^{21}}(-q^7)}
   \sum_{n\in\Zz}
      \frac{(1-q^{42n+7})q^{42n^2+14n+1}}{(1-q^{21n+4})(1-q^{21n+3})}
   -\frac{q(q,q^6,q^7;q^7)_\infty}{(q^2,q^3,q^4,q^5;q^7)_\infty}, \\
\mathcal{F}_2(q)
   :&=
   \sum_{n=0}^{\infty}
      \frac{q^{n(n+1)}}{(q^{n+1};q)_{n+1}} \nonumber \\
   &=
   \frac{2q^5}{\theta_{q^{21}}(-q^{17})}\tW(q^{17};q^{11},q^6;q^{21}) \nonumber \\
   & \quad 
   +\frac{2q^2}{\theta_{q^{21}}(-q^{10})}\tW(q^{10};q^6,q^4;q^{21})
   +\frac{(q^2,q^5,q^7;q^7)_\infty}{(q,q^3,q^4,q^6;q^7)_\infty} \nonumber \\
   &=
   \frac{2}{\theta_{q^{21}}(-q^{17})}
   \sum_{n\in\Zz}
      \frac{(1-q^{42n+17})q^{42n^2+34n+5}}{(1-q^{21n+11})(1-q^{21n+6})} \nonumber \\
   &\quad
   +\frac{2}{\theta_{q^{21}}(-q^{10})}
   \sum_{n\in\Zz}
      \frac{(1-q^{21n+10})q^{21n^2+20n+2}}{(1-q^{21n+6})(1-q^{21n+4})}
   +\frac{(q^2,q^5,q^7;q^7)_\infty}{(q,q^3,q^4,q^6;q^7)_\infty}.
\end{align}

\subsection{Order 8 mock theta functions}

\begin{align}
S_0(q)
   :&=
   \sum_{n=0}^{\infty}
      \frac{q^{n^2}(-q;q^2)_n}{(-q^2;q^2)_n} \nonumber \\
   &=
   \frac{2}{\theta_{q^8}(q^3)}\tW(-q^3;q^3,-1;q^8)
   +q\frac{(q^2;q^2)_\infty^2(q^8;q^8)_\infty^2\theta_{q^8}(q)}{(q^4;q^4)_\infty^2\theta_{q^8}(-q^3)^2} \nonumber \\
   &=
   \frac{2}{\theta_{q^8}(q^3)}
   \sum_{n\in\Zz}
      \frac{(1+q^{16n+3})q^{16n^2+6n}}{(1-q^{8n+3})(1+q^{8n})}
   +q\frac{(q^2;q^2)_\infty^2(q^8;q^8)_\infty^2\theta_{q^8}(q)}{(q^4;q^4)_\infty^2\theta_{q^8}(-q^3)^2}, \\
S_1(q)
   :&=
   \sum_{n=0}^{\infty}
      \frac{q^{n(n+2)}(-q;q^2)_n}{(-q^2;q^2)_n} \nonumber \\
   &=
   -\frac{2q^{-1}}{\theta_{q^8}(q)}\tW(-q;q,-1;q^8)
   +q^{-1}\frac{(q^2;q^2)_\infty^2(q^8;q^8)_\infty^2\theta_{q^8}(q^3)}{(q^4;q^4)_\infty^2\theta_{q^8}(-q)^2} \nonumber \\
   &=
   -\frac{2}{\theta_{q^8}(q)}
   \sum_{n\in\Zz}
      \frac{(1+q^{16n+1})q^{16n^2+2n-1}}{(1+q^{8n})(1-q^{8n+1})}
   +q^{-1}\frac{(q^2;q^2)_\infty^2(q^8;q^8)_\infty^2\theta_{q^8}(q^3)}{(q^4;q^4)_\infty^2\theta_{q^8}(-q)^2}, \\
T_0(q)
   :&=
   \sum_{n=0}^{\infty}
      \frac{q^{(n+1)(n+2)}(-q^2;q^2)_n}{(-q;q^2)_{n+1}} \nonumber \\
   &=
   \frac{q^2}{\theta_{q^8}(q^7)}\tW(-q^7;-q^5,q^2;q^8) \nonumber \\
   &=
   \frac{1}{\theta_{q^8}(q^7)}
   \sum_{n\in\Zz}
      \frac{(1+q^{16n+7})q^{16n^2+14n+2}}{(1-q^{8n+2})(1+q^{8n+5})}, \\
T_1(q)
   :&=
   \sum_{n=0}^{\infty}
      \frac{q^{n(n+1)}(-q^2;q^2)_n}{(-q;q^2)_{n+1}} \nonumber \\
   &=
   -\frac{q^5}{\theta_{q^8}(q^{13})}\tW(-q^{13};-q^{7},q^6;q^{8}) \nonumber \\
   &=
   \frac{-1}{\theta_{q^8}(q^{13})}
   \sum_{n\in\Zz}
      \frac{(1+q^{16n+13})q^{16n^2+26n+5}}{(1-q^{8n+6})(1+q^{8n+7})}, \\
U_0(q)
   :&=
   \sum_{n=0}^{\infty}
      \frac{q^{n^2}(-q;q^2)_n}{(-q^4;q^4)_n} \nonumber \\
   &=
   \frac{2}{\theta_{q^4}(q)}\tW(-q;q,-1;q^4) \nonumber \\
   &=
   \frac{2}{\theta_{q^4}(q)}
   \sum_{n\in\Zz}
      \frac{(1+q^{8n+1})q^{8n^2+2n}}{(1+q^{4n})(1-q^{4n+1})}, \\
U_1(q)
   :&=
   \sum_{n=0}^{\infty}
      \frac{q^{(n+1)^2}(-q;q^2)_n}{(-q^2;q^4)_{n+1}} \nonumber \\
   &=
   -\frac{q^2}{\theta_{q^4}(q^5)}\tW(-q^5;q^3,-q^2;q^4) \nonumber \\
   &=
   \frac{-1}{\theta_{q^4}(q^5)}
   \sum_{n\in\Zz}
      \frac{(1+q^{8n+5})q^{8n^2+10n+2}}{(1+q^{4n+2})(1-q^{4n+3})}, \\
V_0(q)
   :&=
   -1
   +2\sum_{n=0}^{\infty}
      \frac{q^{n^2}(-q;q^2)_n}{(q;q^2)_n} \nonumber \\
   &=
   \frac{2}{\theta_{q^8}(-q^2)}\tW(q^2;q,q;q^8)
   -\frac{(q^2;q^2)_\infty^3(q^4;q^4)_\infty}{(q;q)_\infty^2(q^8;q^8)_\infty} \nonumber \\
   &=
   \frac{2}{\theta_{q^8}(-q^2)}
   \sum_{n\in\Zz}
      \frac{1+q^{8n+1}}{1-q^{8n+1}}q^{16n^2+4n}
   -\frac{(q^2;q^2)_\infty^3(q^4;q^4)_\infty}{(q;q)_\infty^2(q^8;q^8)_\infty}, \\
V_1(q)
   :&=
   \sum_{n=0}^{\infty}
      \frac{q^{(n+1)^2}(-q;q^2)_n}{(q;q^2)_{n+1}} \nonumber \\
   &=
   \frac{q}{\theta_{q^8}(-q^4)}\tW(q^4;q^3,q;q^8) \nonumber \\
   &=
   \frac{1}{\theta_{q^8}(-q^4)}
   \sum_{n\in\Zz}
      \frac{(1-q^{16n+4})q^{16n^2+8n+1}}{(1-q^{8n+1})(1-q^{8n+3})}.
\end{align}

\subsection{Order 10 mock theta functions}

\begin{align}
\phi(q)
   :&=
   \sum_{n=0}^{\infty}
      \frac{q^\frac{n(n+1)}{2}}{(q;q^2)_{n+1}} \nonumber \\
   &=
   \frac{2q}{\theta_{q^{10}}(-q^5)}\tW(q^5;q^3,q^2;q^{10})
   +\frac{(q^2;q^2)_\infty(q^5;q^5)_\infty(q^{10};q^{10})_\infty^2}{\theta_{q^5}(-q^2)\theta_{q^{10}}(-q^2)^2} \nonumber \\
   &=
   \frac{2}{\theta_{q^{10}}(-q^5)}
   \sum_{n\in\Zz}
      \frac{(1-q^{20n+5})q^{20n^2+10n+1}}{(1-q^{10n+2})(1-q^{10n+3})}
   +\frac{(q^2;q^2)_\infty(q^5;q^5)_\infty(q^{10};q^{10})_\infty^2}{\theta_{q^5}(-q^2)\theta_{q^{10}}(-q^2)^2}, \\
\psi(q)
   :&=
   \sum_{n=0}^{\infty}
      \frac{q^\frac{(n+1)(n+2)}{2}}{(q;q^2)_{n+1}} \nonumber \\
   &=
   \frac{2q}{\theta_{q^{10}}(-q^5)}\tW(q^5;q^4,q;q^{10})
   -\frac{(q^2;q^2)_\infty(q^5;q^5)_\infty(q^{10};q^{10})_\infty^2}{\theta_{q^5}(-q)\theta_{q^{10}}(-q^4)^2} \nonumber \\
   &=
   \frac{2}{\theta_{q^{10}}(-q^5)}
   \sum_{n\in\Zz}
      \frac{(1-q^{20n+5})q^{20n^2+10n+1}}{(1-q^{10n+1})(1-q^{10n+4})}
   -q\frac{(q^2;q^2)_\infty(q^5;q^5)_\infty(q^{10};q^{10})_\infty^2}{\theta_{q^5}(-q)\theta_{q^{10}}(-q^4)^2}, \\
X(q)
   :&=
   \sum_{n=0}^{\infty}
      \frac{(-1)^nq^{n^2}}{(-q;q)_{2n}} \nonumber \\
   &=
   -\frac{2q^{-1}}{\theta_{q^5}(1)}\tW(-1;-q,q^{-1};q^5)
   -\frac{(q^5;q^5)_\infty^2\theta_{q^{10}}(-q^3)}{(q^{10};q^{10})_\infty\theta_{q^5}(-q)} \nonumber \\
   &=
   \frac{-2}{\theta_{q^5}(1)}
   \sum_{n\in\Zz}
      \frac{(1+q^{10n})q^{10n^2-1}}{(1-q^{5n-1})(1+q^{5n+1})}
   -\frac{(q^5;q^5)_\infty^2\theta_{q^{10}}(-q^3)}{(q^{10};q^{10})_\infty\theta_{q^5}(-q)}, \\
\chi(q)
   :&=
   \sum_{n=0}^{\infty}
      \frac{(-1)^nq^{(n+1)^2}}{(-q;q)_{2n+1}} \nonumber \\
   &=
   -\frac{2q^2}{\theta_{q^5}(q^5)}\tW(-q^5;-q^3,q^2;q^5)
   +q\frac{(q^5;q^5)_\infty^2\theta_{q^{10}}(-q)}{(q^{10};q^{10})_\infty\theta_{q^5}(-q^2)} \nonumber \\
   &=
   \frac{-2}{\theta_{q^5}(q^5)}
   \sum_{n\in\Zz}
      \frac{(1+q^{10n+5})q^{10n^2+10n+2}}{(1-q^{5n+2})(1+q^{5n+3})}
   +q\frac{(q^5;q^5)_\infty^2\theta_{q^{10}}(-q)}{(q^{10};q^{10})_\infty\theta_{q^5}(-q^2)}.
\end{align}

\section*{Acknowledgments}
We are grateful to Professor Masatoshi Noumi (Rikkyo University) for his valuable suggestions on the bilateral basic hypergeometric functions and Slater's translation formulas. 
This work was supported by JST CREST Grant Number JP19209317 and JSPS KAKENHI Grant Number 21K13808.


\bibliographystyle{amsplain}

\medskip
\begin{flushleft}
Genki Shibukawa \\
Department of Mathematics \\
Graduate School of Science \\
Kobe University \\
1-1, Rokkodai, Nada-ku, Kobe, 657-8501, JAPAN\\
E-mail: g-shibukawa@math.kobe-u.ac.jp
\end{flushleft}

\medskip 
\begin{flushleft}
Satoshi Tsuchimi \\
Department of Mathematics \\
Graduate School of Science \\
Kobe University \\
1-1, Rokkodai, Nada-ku, Kobe, 657-8501, JAPAN\\
E-mail: 183s014s@stu.kobe-u.ac.jp
\end{flushleft}

\end{document}